\documentclass[10pt,twoside]{article}

\usepackage[english]{babel} 
\usepackage{amsmath}
\usepackage{amsthm}
\usepackage{amssymb}
\usepackage{a4wide} 
\usepackage{setspace}
\usepackage[center,tableposition=top,figureposition=bottom]{caption}
\usepackage{pstricks}
\usepackage{pst-plot}
\usepackage{multirow}
\usepackage{tabularx}
\usepackage{wrapfig}
\usepackage{subfigure}

\newcommand{\ld}{\ldots}
\newcommand{\pdot}{\cdot \ld \cdot}
\newcommand{\Z}{\mathbb{Z}}
\newcommand{\Q}{\mathbb{Q}}
\newcommand{\R}{\mathbb{R}}
\newcommand{\C}{\mathbb{C}}
\newcommand{\fp}[1]{\{ #1 \}}
\newcommand{\vect}{\mathbf}
\newcommand{\poch}[2]{( #1 )_{#2}}

\newcommand{\A}{\mathcal{A}}
\newcommand{\al}{\alpha}
\newcommand{\bal}{\boldsymbol \alpha}
\newcommand{\gam}{\boldsymbol \gamma}
\newcommand{\la}{\lambda}
\newcommand{\KA}{K_\A(\bal)}
\newcommand{\sign}{\sigma_\A(\bal)}
\newcommand{\signk}[1]{\sigma_\A(#1 \bal)}
\newcommand{\lat}{\mathbb{L}}
\newcommand{\subs}{\subseteq}

\newcommand{\andd}{\text{ and }}

\newcommand{\entier}[1]{\lfloor #1 \rfloor}

\newcommand{\half}{\frac{1}{2}}
\newcommand{\parset}{(0,1) \cap \Q \setminus \{\half\}}
\newcommand{\e}{\vect{e}}
\newcommand{\x}{\vect{x}}
\newcommand{\va}{\vect{a}}
\newcommand{\trian}{\mathcal{T}}

\newtheorem{thm}{Theorem}[subsection]
\newtheorem{cor}[thm]{Corollary}
\newtheorem{lem}[thm]{Lemma}

\theoremstyle{definition}
\newtheorem{defn}[thm]{Definition}
\newtheorem{rmk}[thm]{Remark}

\begin{document}

\title{Algebraicity of the Appell-Lauricella and Horn hypergeometric functions}
\date{\today}
\author{Esther Bod\thanks{Department of Mathematics, Universiteit Utrecht, The Netherlands.
This work was supported by the Netherlands Organisation for Scientific Research (NWO) under the grant OND1331860.
The author would like to thank Frits Beukers for interesting conversations and suggestions.}}
\maketitle

\begin{abstract}
We extend Schwarz' list of irreducible algebraic Gauss functions to the four classes of Appell-Lauricella functions in several variables and the 14 complete Horn functions in two variables.
This gives an example of a family of functions such that for any number of variables there are infinitely many algebraic functions, namely the Lauricella $F_C$ functions.
\end{abstract}


\section{Algebraic hypergeometric functions}

\subsection{Introduction}\label{subsec:introduction}

The classical Gauss hypergeometric function is
\begin{equation*}
F(a,b,c|z) =\sum_{n \geq 0} \frac{\poch{a}{n} \poch{b}{n}}{\poch{c}{n} n!} z^n,
\end{equation*}
\noindent where $a$, $b$ and $c$ are complex parameters.
Here $\poch{x}{n}$ denotes the Pochhammer symbol defined by $\poch{x}{n} = \frac{\Gamma(x+n)}{\Gamma(x)}$.
There are many generalizations to hypergeometric functions in several variables.
The most well-known are the Lauricella functions, introduced by Lauricella in 1893 (\cite{lauricella_introduction}), and the so-called Horn series, introduced by Horn in 1931 (\cite{horn_introduction}).
A Horn series is a series $\sum_{m \in \Z_{\geq 0}^n} c(\vect{m}) \vect{z}^{\vect{m}}$ such that all $f_i(\vect{m})=\frac{c(\vect{m}+\e_i)}{c(\vect{m})}$ are rational functions in $m$ and $n$.
Here $\vect{z}^\vect{m} = z_1^{m_1} \pdot z_n^{m_n}$ and $\e_i$ is the $i^{\textrm{th}}$ standard basis vector.
Up to multiplication by rational functions and multiplication of the coordinates of $\vect{z}$ by constants, there are 34 series in two variables for which the numerator and denominator of all $f_i$ have degree at most 2.
The 14 complete Horn series are the series for which all degrees are exactly 2.
Four of these are the Appell series $F_1$, $F_2$, $F_3$ and $F_4$. 
These can easily be generalized to any number of variables, which gives the Lauricella functions.
For example, the Lauricella $F_D$ function is given by 
\begin{equation*}
F_D(a,\vect{b},c | \vect{z}) = 
\sum_{\vect{m} \in \Z_{\geq 0}^n} \frac{\poch{a}{|\vect{m}|} \poch{\vect{b}}{\vect{m}}}{\poch{c}{|\vect{m}|} \vect{m}!} \vect{z}^{\vect{m}}, 
\end{equation*}
where $\poch{\x}{\vect{m}}$ is given by $\poch{x_1}{m_1} \pdot \poch{x_n}{m_n}$.
For $n=2$, this is the Appell $F_1$ function.

In 1873, Schwarz found a list of all irreducible algebraic Gauss functions (see~\cite{schwarz_Gaussfunction}).
By irreducible we mean that the monodromy group acts irreducible.
This list has been extended to general one-variable hypergeometric functions $_{p+1} F_p$ by Beukers and Heckman (see~\cite{beukers_heckman_monodromy_nFn-1}), to the Appell-Lauricella functions $F_1$ and $F_D$ by Beazley Cohen, Wolfart and Sasaki (\cite{beazley_cohen_wolfart_algebraic_appell_lauricella}), the Appell functions $F_2$ and $F_4$ by Kato (\cite{kato_Appell_F2}, \cite{kato_Appell_F4}) and the Horn $G_3$ function by Schipper (\cite{schipper_thesis}).
In~\cite{beazley_cohen_wolfart_algebraic_appell_lauricella}, Beazley Cohen and Wolfart also give some results on reducible algebraic $F_2$, $F_3$ and $F_4$ functions.

The goal of this paper is to determine the parameter values for which the Appell-Lauricella and Horn series are non-resonant algebraic functions over $\C(\vect{z})$ or $\C(x,y)$. 
Non-resonance is a condition that is almost equivalent to irreducibility, as will be made precise in the next section.
Note that the parameters of an algebraic function have to be rational, since they determine the local exponents of the series.


\subsection{Some general theory}\label{subsec:background}

In this section, we will recall some results about GKZ-hypergeometric functions and prove some lemmas that will by useful in determining the algebraic functions.
We start with the definition of a GKZ-hypergeometric function: 

\begin{defn}\label{defn:GKZ}
Let $\A=\{\va_1, \ld, \va_N\}$ be a finite subset of $\Z^r$ such that the $\Z$-span of $\va_1, \ld, \va_N$ equals $\Z^r$ and there exists a linear form $h$ on $\R^r$ such that $h(\va_i)=1$ for all $i$. 
We assume that $\A$ is saturated, i.e.\ $\R_{\geq 0} \A \cap \Z^r = \Z_{\geq 0} \A$.
Let $\lat \subs \Z^N$be the lattice of relations in $\A$, i.e.\ $\lat = \{(l_1, \ld, l_N) \in \Z^N \ | \ l_1 \va_1 + \ld + l_N \va_N = 0 \}$.
Furthermore, let $\bal \in \Q^r$ (in general, $\bal$ can be an element of $\C^r$, but we will only consider $\bal \in \Q^r$). 
Denote by $\partial_i$ the operator $\frac{\partial}{\partial z_i}$.
The GKZ-system associated with $\A$ and $\bal$, denoted $H_\A(\bal)$, consists of the equation  
\begin{equation*}
\prod_{l_i<0} \partial_i^{-l_i} \Phi =  \prod_{l_i>0} \partial_i^{l_i} \Phi
\end{equation*}
for each $(l_1, \ld, l_N) \in \lat$, and the Euler equations 
$\va_1 z_1 \partial_1 \Phi + \ld + \va_N z_N \partial_N \Phi = \bal \Phi$.
\end{defn}

For every $\gam$ such that $\gamma_1 \va_1 + \ld + \gamma_N \va_N = \bal$, the system of equations has a formal solution 
\begin{equation}\label{eq:powerseries}
\Phi(z_1, \ld, z_N) = \sum_{(l_1, \ld, l_N) \in \lat} \frac{z_1^{l_1+\gamma_1} \cdot \ld \cdot z_n^{l_N+\gamma_N}}{\Gamma(l_1+\gamma_1+1) \cdot \ld \cdot \Gamma(l_N+\gamma_N+1)}.
\end{equation}
If $I$ is a subset of $\{1, \ld, N\}$ such that $\{\va_i \ | \ i \in I\}$ is a maximal independent set and $\gamma_j \in \Z$ for all $j \not\in I$, then the Laurent series has a positive radius of convergence (see~\cite{stienstra_GKZ_hypergeometric_structures}, section 3). 
We will always choose $\gamma_j=0$ for all $j \not\in I$. \\

Let $C(\A)$ be the real positive cone generated by $\A$, i.e. $C(\A) = \R_{\geq 0} \A$.
Furthermore, let $Q(\A)$ be the convex hull.
In this notation, $\A$ is saturated if $C(\A) \cap \Z^r = \Z_{\geq 0} \A$.
Note that $\Z_{\geq 0} \A \subs C(\A) \cap \Z^r$ for all $\A \subs \Z^r$. 

\begin{defn}\label{defn:resonance}
$H_\A(\bal)$ is called resonant if $\bal+\Z^r$ contains a point in a face of $C(\A)$.
\end{defn}

\begin{thm}\cite[Theorem~2.11]{gkz_euler_integrals_and_hypergeometric_functions}\label{thm:irrreducibility}
If $H_\A(\bal)$ is non-resonant, then it is irreducible.
\end{thm}

The converse is almost true:

\begin{rmk}\label{rmk:reducible_is_resonant}
Let $H_\A(\bal)$ be resonant.
Suppose that for every $i \in \{1, \ld, N\}$ there exists $(l_1, \ld, l_N) \in \lat$ such that $l_i \neq 0$.
Then $H_\A(\bal)$ is reducible.
\end{rmk}

We could not find a proof of this Remark in the literature yet.
However, it will be the subject of an upcoming paper by Beukers.
The condition on the lattice is satisfied for all Appell-Lauricella and Horn function, as will be immediately clear from the exposition in the next sections.
Therefore, one can think of irreducibility as being equivalent to non-resonance.
For the Gauss function and some of the Appell-Lauricella functions, irreducibility conditions can be found in the literature.
For these functions, we will prove the equivalence with non-resonance.

\begin{defn}
Let $\KA = (\bal + \Z^r) \cap C(\A)$.
A point $\vect{p} \in \KA$ is called an apexpoint if for every $\vect{q} \in \KA$ such that $\vect{p} \neq \vect{q}$, it holds that $\vect{p}-\vect{q} \not\in C(\A)$.
The number of apexpoints is called the signature of $\A$ and $\bal$ and is denoted by $\sign$.
\end{defn}

Note that $\sign$ only depends on the fractional part $\fp{\bal}$ of $\bal$ (where $\fp{\bal}_i = \fp{\al_i} = \al_i - \entier{\al_i}$).

\begin{lem}\label{lem:apexpoints}
Let $\vect{p} \in \KA$. Then $\vect{p}$ is an apexpoint if and only if $\vect{p}-\va_i \not\in C(\A)$ for all $\va_i \in \A$.
\end{lem}

\begin{proof}
If there exists $\va_i \in \A$ such that $\vect{p}-\va_i \in C(\A)$, then we can take $\vect{q}=\vect{p}-\va_i \in \KA$.
Then $\vect{p} \neq \vect{q}$ and $\vect{p}-\vect{q} = \va_i  \in C(\A)$, so $\vect{p}$ is not an apexpoint. 

Suppose that $\vect{p} \in \KA$ is not an apexpoint.
Then there exists $\vect{q} \in \KA$ such that $\vect{p} \neq \vect{q}$ and $\vect{p}-\vect{q} \in C(\A)$.
Since $\vect{q} \in C(\A)$, there exists $\la_1, \ld, \la_N \geq 0$ such that $\vect{q} = \la_1 \va_1 + \ld + \la_N \va_N$. 
Define $\vect{v}=\vect{p}-\vect{q}$. 
Then $\vect{v} \in C(\A) \cap \Z^r$.
$\A$ is saturated, so there exist $\mu_1, \ld, \mu_N \in \Z_{\geq 0}$ such that $\vect{v} =\mu_1 \va_1 + \ld + \mu_N \va_N$.
It follows that $\vect{p} = \vect{q} + \vect{v} = \sum_{i=1}^N (\la_i + \mu_i) \va_i$.
Since $\vect{v} \neq 0$, there is some $i$ such that $\mu_i \geq 1$.
Now consider $\vect{p}-\va_i$. 
This is clearly an element of $\bal + \Z^r$, and since $\la_i + \mu_i - 1 \geq 0$, it also lies in $C(\A)$.
Hence $\vect{p}-\va_i \in \KA$.
\end{proof}

\begin{lem}\cite[Proposition~1.9]{beukers_algebraic_Ahypergeometric_functions}\label{lem:maximal_signature}
Then $\sign$ is less than or equal to the simplex volume of $Q(\A)$.
\end{lem}

The simplex volume is a normalization of the Euclidean volume, such that the simplex spanned by the standard basis has volume 1.
Since the determinant function is the only linear function that gives 0 if two of the arguments are equal and maps the standard basis to 1, the simplex volume of the simplex spanned by the vectors $\vect{v}_1, \ld, \vect{v}_n$ equals $|\det(\vect{v}_1, \ld, \vect{v}_n)|$.

\begin{lem}
\label{lem:signature_entier_linear_forms}
Let $H_\A(\bal)$ be non-resonant and 
$C(\A) = \{ \x \in \R^r \ | \ m_1(\x) \geq 0, \ld, m_d(\x) \geq 0 \}$
where $m_1, \ld, m_d$ are linear forms with integral coefficients.
Then $\sign$ only depends on $(\entier{m_1(\bal)}, \ld, \entier{m_d(\bal)})$, but not on $\bal$ itself.
\end{lem}

\begin{proof}
Let $\x \in \R^r$.
Then $\x + \bal$ is an apexpoint if and only if $\x+\bal \in C(\A)$ and for all $i$, $\x-\va_i+\bal \not\in C(\A)$.
Equivalently, we have $m_j(\x) \geq -m_j(\bal$) for all $j$, and for all $i$ there exists $j$ such that $m_j(\x) < m_j(\va_i)-m_j(\bal)$.
Since $m_j(\x)$ and $m_j(\va_i)$ are integral, whereas $m_j(\bal)$ is non-integral, the apexpoints are those $\x+\bal$ such that we have $m_j(\x) \geq -\entier{m_j(\bal)}$ for all $j$, and for all $i$ there exists $j$ such that $m_j(\x) \leq m_j(\va_i)-\entier{m_j(\bal)}-1$.
Hence the conditions on $\x+\bal$ to be an apexpoint only depend on $\entier{m_j(\bal)}$.
\end{proof}

The following Theorem will be our main tool to classify the algebraic hypergeometric functions.
It reduces the problem of finding algebraic functions to a combinatorical problem.

\begin{thm}\cite[Theorem~1.10]{beukers_algebraic_Ahypergeometric_functions}\label{thm:algebraic_solutions}
Suppose that $H_\A(\bal)$ is non-resonant. 
Let $D$ be the smallest common denominator of the coordinates of $\bal \in \Q^r$. 
Then the solutions of $H_\A(\bal)$ are algebraic over $\C(\vect{z})$ if and only if $\signk{k}$ equals the simplex volume of $Q(\A)$ for all integers $k$ with $1 \leq k < D$ and $\gcd(k,D)=1$.
\end{thm}

\begin{rmk}\label{rmk:all_solutions_algebraic}
Note that either all solutions of $H_\A(\bal)$ are algebraic or they are all transcendental.
This follows from the fact that the solutions are algebraic if and only if the monodromy group is finite. 
\end{rmk}

\begin{cor}\label{cor:algebraic_orbits}
Let $\bal \in \Q^r$.
If $H_\A(\bal)$ is non-resonant, then algebraicity of the solutions of the GKZ system only depends on $\A$ and $\fp{\bal}$.
Furthermore, either the solution set of $H_\A(k \bal)$ consists of algebraic functions for all $k$ coprime to the smallest common denominator of the coordinates of $\bal$, or the solutions are transcendental for all $k$.
\end{cor}

By this Corollary, it suffices to consider $\bal$ such that $\bal=\fp{\bal}$.
Troughout this paper, we will assume that $0 \leq \al_i < 1$.

\begin{rmk}\label{rmk:algorithm_interlacing}
It is well known that an irreducible Gauss function $F(a,b,c|z)$ is algebraic if and only if for every $k$ coprime with the denominators of $a, b$ and $c$, we have either $\fp{ka}  \leq \fp{kc} < \fp{kb}$ or $\fp{kb} \leq \fp{kc} < \fp{ka}$ (see Theorem~\ref{thm:gauss_irreducible_algebraic}).
Using Lemma~\ref{lem:signature_entier_linear_forms} and Theorem~\ref{thm:algebraic_solutions}, we can find such an interlacing condition for other algebraic hypergeometric functions.
It clearly suffices to find a condition on $\bal$ to have maximal signature.
As input we need the linear forms $m_i$ that determine the faces of the cone $C(\A)$.
Write $m_i(\x) = \sum_j m_{ij} x_j$.
Since we only have to consider $\bal$ such that $\al_i \in [0,1)$ for all $i$, $\entier{m_i(\bal)}$ can only take integral values between $\sum_j \min(m_{ij},0)$ and $\sum_j \max(m_{ij},0)$ (both boundaries are excluded, unless they are zero).
Hence $(\entier{m_1(\bal)}, \ld, \entier{m_d(\bal)})$ takes only finitely many values.
For each of those, it suffices to find one corresponding $\bal$ and compute the number of apexpoints.
Finding $\bal$ boils down to solving a linear system of inequalities.
This can easily be done by hand or using a computer algebra system, which will also detect the values of $(\entier{m_1(\bal)}, \ld, \entier{m_d(\bal)})$ for which no $\bal$ exists.
Having found $\bal$, finding apexpoints can again be done by solving a system of linear inequalities, in this case over the integers.

Note that finding the interlacing condition can entirely be done by a computer.
However, this algorithm can be a bit slow, although in turns out to be fast enough for the functions considered in this paper. 
\end{rmk}

The following Remark makes it possible to reduce the number of variables of a function: 

\begin{rmk}\label{rmk:reduction_algebraicity}
If $f(z_1, \ld, z_n)$ is an algebraic function over $\C(z_1, \ld, z_n)$ and $a \in \C$, then $f(z_1, \ld, z_{i-1}, a, z_{i+1}, \ld, z_n)$ is algebraic over $\C(z_1, \ld, z_{i-1}, z_{i+1}, \ld, z_n)$ for all $i \in \{1,\ld,n\}$.
\end{rmk}

To prove that certain functions are not algebraic, we will have to find $k$ such that $\gcd(k,D)=1$ and $\signk{k}$ is not maximal.
By the following Lemma, it suffices to find $k$ coprime with the denominator of several, but not all, coefficients.
Sometimes we can show that $\signk{k}$ is not maximal if some parameter is close to $\half$.
The second Lemma handles this case.

\begin{lem}\label{lem:higher_denom}
Let $k, D$ and $\tilde{D}$ be positive integers such that $D | \tilde{D}$ and $\gcd(k,D)=1$.
Then there exists an integer $l$ such that $l \equiv k \pmod D$ and $\gcd(l,\tilde{D})=1$.
\end{lem}

\begin{proof}
Write $D=p_1^{k_1} \pdot p_r^{k_r}$ and $\tilde{D}=p_1^{l_1} \pdot p_r^{l_r} q_1^{m_1} \pdot q_s^{m_s}$.
Define $c=p_1^{l_1} \pdot p_r^{l_r}$.
Then there exists an integer $E$ such that $DE \equiv 1 \pmod{\frac{\tilde{D}}{c}}$.
Define $l=k-(k-1)DE$. 
Then it is clear that $l \equiv k \pmod D$ and $l \equiv 1 \pmod{\frac{\tilde{D}}{c}}$, so $\gcd(l,\frac{\tilde{D}}{c})=1$.
This implies that $l$ is coprime with $D \frac{\tilde{D}}{c}$.
Now it follows from the definition of $c$ that $l$ is coprime with $\tilde{D}$.
\end{proof}

\begin{lem}\label{lem:closetohalf}
Let $r=\frac{p}{q}$ with $\gcd(p,q)=1$ and $q \geq 3$.
Define $d=1$ if $q$ is odd, $d=2$ if 4 divides $q$ and $d=4$ if $q \equiv 2 \pmod 4$.
Let $t \in (0,\half)$.
If $q \geq \frac{d}{1-2t}$, then there exists $k \in \Z$ with $\gcd(k,q)=1$ such that $\fp{kr} \in [t,\half)$.
\end{lem}

\begin{proof}
Choose $p' \in \Z$ such that $p p' \equiv 1 \pmod q$ and take $k=\frac{q-d}{2}p'$.
Then $k$ is an integer with $\gcd(k,q)=1$.
Since $q \geq \frac{d}{1-2t}$, we have $\frac{q-d}{2q} \geq t$.
Hence $\fp{kr} = \fp{\frac{q-d}{2} \cdot \frac{p p'}{q}} = \fp{\frac{q-d}{2q}} \in [t,\half)$.
\end{proof}

\begin{rmk}\label{rmk:start_from_series}
A GKZ-function is determined by the set $\A$ and the parameter vector $\bal$.
We will take the opposite approach: we will start with a Laurent series of the form \eqref{eq:powerseries} and read off $\lat$ and $\gam$.
Then we choose $\A$ such that $\lat$ is the lattice of relations in $\A$, and compute $\bal = \sum_{i=1}^N \gamma_i \va_i$.
Of course, we have to choose $\A$ such that it spans $\Z^r$, lies in a hyperplane $h(\x)=1$ and is saturated.
\end{rmk}

\begin{rmk}\label{rmk:transformations}
Let $\A, \cal{B} \subs$ $\Z^r$ be as in Definition~\ref{defn:GKZ} with $|\A|=|\cal{B}|$, and let $f: \Z^r \rightarrow \Z^r$ be a linear isomorphism such that $f(\A) = \cal{B}$ (as sets, so the order of the vectors can be changed).
Write $\boldsymbol\beta = f(\bal)$.
Then $C(\cal{B}) =$ $f(C(\A))$ and $\sigma_{\cal B}(\boldsymbol \beta) = \sign$, so $H_{\cal B}(\boldsymbol\beta)$ is non-resonant and algebraic if and only if $H_\A(\bal)$ is non-resonant and algebraic.
Hence if the faces of $C(\A)$ are given by linear forms $m_i$ as in Lemma~\ref{lem:signature_entier_linear_forms}, then the faces of $C(\cal{B})$ are given by $m_i \circ f^{-1}$ and the algebraic solutions of $H_{\cal B}(\boldsymbol\beta)$ have parameters $f(\bal)$, where $\bal$ is chosen so that $H_\A(\bal)$ has algebraic solutions.
\end{rmk}


To show that a set $\A$ satisfies the conditions of Definition~\ref{defn:GKZ} and to compute the simplex volume of $Q(\A)$, we will use triangulations of $Q(\A)$.

\begin{defn}\label{defn:triangulation}
Let $\A \subs \Z^r$ be as in Definition~\ref{defn:GKZ}.
A triangulation of $Q(\A)$ is a finite set $\trian = \{Q(V_1), \ld, Q(V_l)\}$ such that each $V_i$ is a subset of $\A$ consisting of $r$ linearly independent elements, $Q(V_i) \cap Q(V_j) = Q(V_i \cap V_j)$ for all $i$ and $j$ and $Q(\A) = \cup_{i=1}^l Q(V_i)$.
If all $Q(V_i)$ have simplex volume 1, then the triangulation is called unimodular.
\end{defn}

\begin{rmk}\label{rmk:Q=C&h1}
Note that $C(V) = \R_{\geq 0} Q(V)$ and $Q(V) = C(V) \cap h^{-1}(\{1\})$ for all $V \subs \A$.
\end{rmk}

\begin{lem}\label{lem:convex_hull}
Suppose that $V_1, \ld, V_l$ are subsets of $\A$ consisting of $r$ linearly independent vectors with determinant $\pm 1$, such that $\A = \cup_{i=1}^l V_i$, $C(V_i) \cap C(V_j) \subs C(V_i \cap V_j)$ for all $i$ and $j$ and $\cup_{i=1}^l C(V_i)$ is convex.
Then $\trian = \{Q(V_1), \ld, Q(V_l)\}$ is a unimodular triangulation of $Q(\A)$.
\end{lem}

\begin{proof}
Since the determinant of the vectors in $V_i$ is $\pm 1$, the vectors are linearly independent and $Q(V_i)$ has volume 1.
It is clear that $C(V_i \cap V_j) \subs C(V_i) \cap C(V_j)$, so $C(V_i) \cap C(V_j) = C(V_i \cap V_j)$.
Now Remark~\ref{rmk:Q=C&h1} implies that $Q(V_i) \cap Q(V_j) = Q(V_i \cap V_j)$.
It is also clear that $\cup_{i=1}^l Q(V_i) \subs Q(\A)$.
Note that $\cup_{i=1}^l Q(V_i) = \cup_{i=1}^l C(V_i) \cap h^{-1}(\{1\})$ is a convex set.
It contains $\A = \cup_{i=1}^l V_i$, so it also contains the convex hull of $\A$.
This implies that $Q(\A) \subs \cup_{i=1}^l Q(V_i)$.
\end{proof}

\begin{lem}\label{lem:properties_trian}
Let $\trian=\{Q(V_1), \ld, Q(V_l)\}$ be a unimodular triangulation of $Q(\A)$.
Then \\
\textit{(i)} $vol(Q(\A))=l$. \\
\textit{(ii)} $C(\A) = \cup_{i=1}^l C(V_i)$. \\
\textit{(iii)} $\A$ is saturated. 
\end{lem}

\begin{proof}
\textit{(i)} is clear.
\textit{(ii)} follows from Remark~\ref{rmk:Q=C&h1}. 

For \textit{(iii)}, let $\x \in \R_{\geq 0} \A \cap \Z^r$.
Then $\x \in C(\A)$, so there exists $V_i =\{\vect{v}_1, \ld, \vect{v}_r\} \subs \A$ such that $\x \in C(V_i)$.
Then there exist $\la_1, \ld, \la_r \geq 0$ such that $\x = \sum_{i=1}^r \la_i \vect{v}_i$.
Hence $\x = (\vect{v}_1, \ld, \vect{v}_r) \boldsymbol{\la}$, so $\boldsymbol{\la} = (\vect{v}_1, \ld, \vect{v}_r)^{-1} \x$ (here $(\vect{v}_1, \ld, \vect{v}_r)$ denotes the matrix with columns $\vect{v}_1, \ld, \vect{v}_r$).
Since the determinant of $(\vect{v}_1, \ld, \vect{v}_r)$ is $\pm 1$, $\boldsymbol{\la}$ has integral coordinates.
They are non-negative because $\la_i \geq 0$.
It follows that $\x = (\vect{v}_1, \ld, \vect{v}_r) \cdot \boldsymbol{\la}$ is an element of $\Z_{\geq 0} \A$. 
\end{proof}


\subsection{The Gauss hypergeometric function}\label{subsec:Gauss}

The Appell-Lauricella and Horn functions are generalizations of the Gauss hypergeometric function $F(a,b,c|z) = \sum_{n \geq 0} \frac{\poch{a}{n} \poch{b}{n}}{\poch{c}{n} n!} z^n$. 
This function can be reproduced as a GKZ-function as follows:
take $\A=\{\e_1, \e_2, \e_3, \e_1+\e_2-\e_3\} \subs \Z^3$ and $\bal=(-a,-b,c-1)$.
With $\gam=(-a,-b,c-1,0)$, we get the formal Laurent series solution
\begin{equation*}
G(a,b,c|\vect{z}) = 
z_1^{-a} z_2^{-b} z_3^{c-1} \sum_{n \in \Z} \frac{(z_1^{-1} z_2^{-1} z_3 z_4)^n}{\Gamma(1-n-a) \Gamma(1-n-b) \Gamma(c+n) \Gamma(n+1)}.
\end{equation*}
The Laurent series converges because $\gamma_4 \in \Z$.
Since $\Gamma(n+1)$ has a pole for $n<0$, we only sum over $n \geq 0$.
Note that $F(z) = G(1,1,1,z)$ and $G(\vect{z}) = z_1^a z_2^b z_3^{1-c} F(z_1^{-1} z_2^{-1} z_3 z_4)$.
Hence by Remark~\ref{rmk:reduction_algebraicity}, the Gauss function $F(a,b,c|z)$ is algebraic if and only if the GKZ-function $G(a,b,c|\vect{z})$ is algebraic. 
In the remaining of this paper, we will identify the Appell-Lauricella and Horn function with their GKZ-counterpart. \\

To find all algebraic Appell-Lauricella and Horn functions, we will use a reduction to algebraic Gauss functions.
In 1873, Schwarz published a list of all irreducible algebraic Gauss functions (\cite{schwarz_Gaussfunction}).
To each function $F(a,b,c,|z)$, associate the triple $(\la, \mu, \nu) = (1-c, c-a-b, b-a)$.
Up to permutations of $\{\la, \mu, \nu\}$, sign changes of each of $\la$, $\mu$ and $\nu$ and addition of $(l,m,n) \in \Z^3$ with $l+m+n$ even to $(\la, \mu, \nu)$, Table~\ref{tab:gauss_lmn} gives all irreducible algebraic Gauss functions.

\renewcommand{\arraystretch}{2}
\begin{table}[htb]
\begin{center}
\caption{The tuples $(\la,\mu,\nu)$ such that $F(a,b,c|z)$ is irreducible and algebraic} 
\label{tab:gauss_lmn} 
\begin{tabular}{lllllll}
\hline 
$(\half, \half, s)$ & \multicolumn{2}{l}{with $s \in \Q \setminus \Z$} \\

$(\half, \frac{1}{3}, \frac{1}{3})$ & 
$(\half, \frac{1}{3}, \frac{1}{4})$ & 
$(\half, \frac{1}{3}, \frac{1}{5})$ & 
$(\half, \frac{2}{5}, \frac{1}{3}$) & 
$(\half, \frac{2}{5}, \frac{1}{5})$ & 
$(\frac{2}{3}, \frac{1}{3}, \frac{1}{3})$ &
$(\frac{2}{3}, \frac{1}{3}, \frac{1}{5})$ \\ 

$(\frac{2}{3}, \frac{1}{4}, \frac{1}{4})$ &
$(\frac{2}{3}, \frac{1}{5}, \frac{1}{5})$ &
$(\frac{2}{5}, \frac{1}{3}, \frac{1}{3})$ & 
$(\frac{2}{5}, \frac{2}{5}, \frac{2}{5})$ &
$(\frac{3}{5}, \frac{1}{3}, \frac{1}{5})$ & 
$(\frac{3}{5}, \frac{2}{5}, \frac{1}{3})$ &
$(\frac{4}{5}, \frac{1}{5}, \frac{1}{5})$ \\ 
\hline
\end{tabular}
\end{center}
\end{table}
\renewcommand{\arraystretch}{1}

To compute all triples $(a,b,c)$ such that $F(a,b,c|z)$ is irreducible and algebraic, note that $(a,b,c) = (\frac{1-\la-\mu-\nu}{2}, \frac{1-\la-\mu+\nu}{2}, 1-\la)$. 
Since algebraicity only depends on the fractional part of $(a,b,c)$, it suffices to compute $(a,b,c)$ for all 48 tuples obtained from $(\la, \mu, \nu)$ by permutations and changes of signs and choose $(l,m,n)$ such that $0 \leq a,b,c < 1$.
For the tuple $(\la, \mu, \nu)=(\half, \half, s)$, this give 3 possible forms for $(a,b,c)$:
writing $r=\fp{a}$, we have $(a,b,c) = (r, 1-r, \half), (r, r+\half, \half)$ or $(r, r+\half, 2r) \pmod \Z$.
Since $s \in \Q \setminus \Z$, we have $r \in \parset$.
We will call these tuples Gauss triples of type 1. 
This terminology is not standard, but will be used troughout this paper. 
For the other 14 tuples $(\la, \mu, \nu)$, we use the computer to compute $(a,b,c)$ for all 48 tuples obtained from $(\la, \mu, \nu)$.
This gives 408 triples $(a,b,c)$ which form orbits under conjugation by Corollary~\ref{cor:algebraic_orbits}.
By conjugation we mean the action of $(\Z/D\Z)^*$ by $k(\fp{a},\fp{b},\fp{c})= (\fp{ka}, \fp{kb}, \fp{kc})$, where $D$ is the smallest common denominator of $a$, $b$ and $c$.
Furthermore, the tuples come in pairs $(a,b,c)$ and $(b,a,c)$.
In Table~\ref{tab:gauss_abc_orbits}, the smallest element of each pair of orbits is given (where $\frac{p}{q}$ is considered to be smaller than $\frac{u}{v}$ if either $q<v$, or $q=v$ and $p \leq u$.
Tuples of fractions are ordered lexicographically).
We call these 408 tuples Gauss tuples of type 2.
Notice that the denominators of $a$ and $c$ are at most 60 and 5, respectively, for all tuples of type 2.

\renewcommand{\arraystretch}{2}
\begin{table}[htb]
\begin{center}
\caption{The tuples $(a,b,c)$ such that $F(a,b,c|z)$ is irreducible and algebraic} 
\label{tab:gauss_abc_orbits} 
\begin{tabularx}{\textwidth}{XXXXXXX} 
\hline
$(r, -r, \half)$ &
$(r, r+\half, \half)$ &
\multicolumn{2}{l}{$(r, r+\half, 2r)$} &
\multicolumn{3}{l}{with $r \in \parset$} \\

$(\half, \frac{1}{6}, \frac{1}{3})$ &
$(\frac{1}{4}, \frac{3}{4}, \frac{1}{3})$ & 
$(\frac{1}{4}, \frac{7}{12}, \half)$ & 
$(\frac{1}{4}, \frac{7}{12}, \frac{1}{3})$ &
$(\frac{1}{6}, \frac{5}{6}, \frac{1}{3})$ &
$(\frac{1}{6}, \frac{5}{6}, \frac{1}{4})$ &
$(\frac{1}{6}, \frac{5}{6}, \frac{1}{5})$ \\

$(\frac{1}{6}, \frac{5}{12}, \frac{1}{3})$ &
$(\frac{1}{6}, \frac{5}{12}, \frac{1}{4})$ &
$(\frac{1}{6}, \frac{11}{30}, \frac{1}{3})$ &
$(\frac{1}{6}, \frac{11}{30}, \frac{1}{5})$ &
$(\frac{1}{10}, \frac{3}{10}, \frac{1}{5})$ &
$(\frac{1}{10}, \frac{9}{10}, \frac{1}{3})$ &
$(\frac{1}{10}, \frac{9}{10}, \frac{1}{5})$ \\

$(\frac{1}{10}, \frac{13}{30}, \frac{1}{3})$ &
$(\frac{1}{10}, \frac{13}{30}, \frac{1}{5})$ &
$(\frac{1}{12}, \frac{5}{12}, \frac{1}{4})$ &
$(\frac{1}{12}, \frac{7}{12}, \frac{1}{3})$ &
$(\frac{1}{15}, \frac{7}{15}, \frac{1}{3})$ &
$(\frac{1}{15}, \frac{7}{15}, \frac{1}{5})$ &
$(\frac{1}{15}, \frac{11}{15}, \frac{1}{5})$ \\

$(\frac{1}{15}, \frac{11}{15}, \frac{3}{5})$ &
$(\frac{1}{20}, \frac{11}{20}, \frac{1}{5})$ &
$(\frac{1}{20}, \frac{11}{20}, \frac{2}{5})$ &
$(\frac{1}{20}, \frac{13}{20}, \half)$ &
$(\frac{1}{20}, \frac{13}{20}, \frac{1}{5})$ &
$(\frac{1}{24}, \frac{13}{24}, \frac{1}{3})$ &
$(\frac{1}{24}, \frac{13}{24}, \frac{1}{4})$ \\

$(\frac{1}{24}, \frac{17}{24}, \half)$ &
$(\frac{1}{24}, \frac{17}{24}, \frac{1}{4})$ &
$(\frac{1}{24}, \frac{19}{24}, \half)$ &
$(\frac{1}{24}, \frac{19}{24}, \frac{1}{3})$ &
$(\frac{1}{30}, \frac{11}{30}, \frac{1}{5})$ &
$(\frac{1}{30}, \frac{19}{30}, \frac{1}{3})$ &
$(\frac{1}{60}, \frac{31}{60}, \frac{1}{3})$ \\

$(\frac{1}{60}, \frac{31}{60}, \frac{1}{5})$ &
$(\frac{1}{60}, \frac{41}{60}, \half)$ &
$(\frac{1}{60}, \frac{41}{60}, \frac{1}{5})$ &
$(\frac{1}{60}, \frac{49}{60}, \half)$ &
$(\frac{1}{60}, \frac{49}{60}, \frac{1}{3})$ \\
\hline
\end{tabularx}
\end{center}
\end{table}
\renewcommand{\arraystretch}{1}

We recall the usual criterion for irreducibility and the interlacing condition for the Gauss function (see e.g.~\cite{beukers_heckman_monodromy_nFn-1}).
These are special cases of Corollary~\ref{cor:FD_irreducible} and Lemma~\ref{lem:FD_interlacing}.

\begin{thm}\label{thm:gauss_irreducible_algebraic}
$F(a,b,c|z)$ is irreducible if and only if $a$, $b$, $c-a$ and $c-b$ are non-integral.
If it is irreducible, then it is algebraic if and only if for every $k$ coprime with the denominators of $a, b$ and $c$, either $\fp{ka}  \leq \fp{kc} < \fp{kb}$ or $\fp{kb} \leq \fp{kc} < \fp{ka}$.
\end{thm}


\section{The Appell-Lauricella functions}

\subsection{The Appell $F_1$ and Lauricella $F_D$ functions}\label{subsec:F1}

The Lauricella $F_D$ function is defined by 
\begin{equation*}
F_D(a,\vect{b},c | \vect{z}) = 
\sum_{\vect{m} \in \Z_{\geq 0}^n} \frac{\poch{a}{|\vect{m}|} \poch{\vect{b}}{\vect{m}}}{\poch{c}{|\vect{m}|} \vect{m}!} \vect{z}^{\vect{m}}. 
\end{equation*}
Up to a constant factor, this equals
\begin{equation*}
\sum_{\vect{m} \in \Z_{\geq 0}^n} \frac{z_1^{m_1} \cdot \ld \cdot z_n^{m_n}}
{\Gamma(1-m_1-\ld-m_n-a) \prod_{i=1}^n \Gamma(1-m_i-b_i) \Gamma(m_1 + \ld + m_n+c) \prod_{i=1}^n \Gamma(1+m_i)}.
\end{equation*}
Hence the lattice is 
\begin{equation*}
\lat = \bigoplus_{i=1}^n \Z(-\e_1+\e_{i+1}+\e_{n+2}+\e_{n+i+2}) \subs \Z^{2n+2}
\end{equation*}
and $\gam = (-a, -\vect{b}, c-1, \vect{0}) \in \R^{2n+2}$.
We can take
\begin{equation*}
\A=\{\e_1, \e_2, \ld, \e_{n+2}, \e_1+\e_2-\e_{n+2}, \e_1+\e_3-\e_{n+2}, \ld, \e_1+\e_{n+1}-\e_{n+2}\} \subs \Z^{n+2}
\end{equation*}
and $\bal = \sum_{i=1}^{2n+2} \gamma_i \va_i = (-a, -\vect{b}, c-1) \in \Q^{n+2}$. 

\begin{lem}\label{lem:FD_normality}
$\A$ is saturated.
\end{lem}

\begin{proof}
Suppose that $\x \in \Z^{n+2}$ lies in the $\R_{\geq 0}$-span of $\A$.
Then there exist $\mu_1, \ld, \mu_{2n+2} \in \R_{\geq 0}$ such that $\x = \sum_{i=1}^{2n+2} \mu_i \va_i$. 
Define $\la_{n+3}=\min(x_1,x_2)$, and for $n+4 \leq i \leq 2n+2$, define recursively $\la_i = \min(x_1-\la_{n+3}- \ld - \la_{i-1}, x_{i-n-1})$.
Furthermore, for $2 \leq i \leq n+1$, let $\la_i=x_i-\la_{n+i+1}$.
Finally, let $\la_1=x_1-\la_{n+3} - \ld - \la_{2n+2}$ and $\la_{n+2} = x_{n+2} + \la_{n+3} + \ld + \la_{2n+2}$.
Then it is clear that $\la_i \in \Z$ for all $i$, and $\sum_{i=1}^{2n+2} \la_i \va_i = \x$.
It remains to show that $\la_i \geq 0$ for all $i$. 
This is clear for all $i$ except for $i=n+2$.
If $\la_{n+i+1} = x_i$ for all $2 \leq i \leq n+1$, then $\la_{n+2} = x_2 + \ld + x_{n+1} + x_{n+2} \geq 0$.
If there exists $2 \leq i \leq n+1$ such that $\la_{n+i+1} = x_1-\la_{n+3} - \ld - \la_{n+i}$, then $\la_{n+j}=0$ for all $i+2 \leq j \leq 2n+2$, so $\la_{n+2} =  \la_{n+3} + \ld + \la_{n+i} + x_1 -\la_{n+3} - \ld - \la_{n+i} + x_{n+2} = x_1 + x_{n+2} \geq 0$.
It follows that $\x$ lies in the $\Z_{\geq 0}$-span of $\A$.
\end{proof}

\begin{lem}\label{lem:FD_CA}
The positive real cone spanned by $\A$ is
\begin{equation*}
C(\A) = \{\x \in \R^{n+2} \ | \ x_1, \ld, x_{n+1} \geq 0, x_1+x_{n+2} \geq 0, x_2 + \ld + x_{n+1}+x_{n+2} \geq 0\}.
\end{equation*}
\end{lem}

\begin{proof}
It is clear that $C(\A)$ is included in this set.
Suppose that $x_1, \ld, x_{n+1} \geq 0$, $x_1+x_{n+2} \geq 0$ and $x_2 + \ld + x_{n+1}+x_{n+2} \geq 0$.
Define $\la_i$ as in the proof of Lemma~\ref{lem:FD_normality}.
Then it is clear that $\x = \sum_{i=1}^{2n+2} \la_i \va_i$, and by an argument very similar to the argument given in the proof of Lemma~\ref{lem:FD_normality}, it follows that all $\la_i$ are non-negative.
Hence $\x \in C(\A)$.
\end{proof}

\begin{cor}\label{cor:FD_irreducible}
$F_D(a,\vect{b},c | \vect{z})$ is non-resonant if and only if $a$, $b_1,\ld,b_n$, $c-a,$ and $c-b_1-\ld-b_n$ are non-integral.
\end{cor}

\begin{rmk}\label{rmk:FD_irreducible_is_nonresonant}
In Proposition~1 of \cite{sasaki_Lauricella_FD}, Sasaki shows that $F_D(a,\vect{b},c | \vect{z})$ function is irreducible if and only if $a$, $b_1,\ld,b_n$, $c-a,$ and $c-b_1-\ld-b_n$ are non-integral.
Hence in this case, non-resonance is equivalent to irreducibility.
\end{rmk}

To compute the simplex volume of the convex hull of $\A$, we map $\A$ to the hyperplane $v_{n+2}=1$ by the invertible transformation $\vect{v} \mapsto (v_1, \ld, v_{n+1}, v_1+\ld+v_{n+2})$.
Now we can omit the last coordinate.
This gives the set $\tilde{\A} = \{\e_1, \ld, \e_{n+1}, 0, \e_1+\e_2, \ld, \e_1+\e_{n+1}\} \subs \Z^{n+1}$.

\begin{lem}\label{lem:FD_convexhull}
The convex hull of $\tilde{\A}$ is $\{\x \in \R^{n+1} \ | \ 0 \leq x_1, \ld, x_{n+1} \leq 1 , 0 \leq x_2 + \ld + x_{n+1} \leq 1\}$.
\end{lem}

\begin{proof}
Denote the above set by $V$.
It is clear that $V$ is a convex set containing $\tilde{\A}$, so $Q(\tilde{\A})$ is contained in $V$. 
Let $\x \in V$. 
For $i \neq n+2$, define $\la_i$ as in the proof of Lemma~\ref{lem:FD_normality}.
Similar to Lemma~\ref{lem:FD_normality}, we have $\la_i \geq 0$ for all $i \neq n+2$.
Note that $\x = \sum_{i=1}^{2n+2} \la_i \va_i$ for every value of $\la_{n+2}$.
By an argument similar to the proof of Lemma~\ref{lem:FD_normality}, one can show that $\la_1 + \ld + \la_{n+1} + \la_{n+3} + \ld + \la_{2n+2}$ is equal to either $x_1$ or $x_2+\ld+x_{n+1}$.
In both cases it is smaller than 1, so we can define $\la_{n+2}=1-(\la_1 + \ld + \la_{n+1} + \la_{n+3} + \ld + \la_{2n+2}) \geq 0$.
\end{proof}

The volume of $Q(\A)$ can now be obtained by an $(n+1)$-fold integration.

\begin{cor}\label{cor:FD_apexpoints}
The simplex-volume of $Q(\A)$ is $n+1$, so there are at most $n+1$ apexpoints. 
\end{cor}

\begin{lem}\label{lem:FD_interlacing}
Suppose that $F_D(a,\vect{b},c | \vect{z})$ is non-resonant.
Then there are $n+1$ apexpoints if and only if either
$\fp{c} < \fp{a}$ and $\fp{b_1} + \ld + \fp{b_n} \leq \fp{c}$, or $\fp{a} \leq \fp{c}$ and $\fp{c}+n-1 < \fp{b_1} + \ld + \fp{b_n}$.
\end{lem}

\begin{proof}
Let $\vect{p} \in \R^{n+2}$.
Since $\vect{p}$ is an apexpoint if and only if $\vect{p} \in \KA$ and $\vect{p} - \va_i \not\in C(\A)$ for all $\va_i \in \A$, the apexpoints are precisely the points $\vect{p} = \x + \bal$, with $\x \in \Z^{n+2}$, satisfying the following conditions: 
$x_1, \ld, x_{n+1} \geq 0$, $x_1 + x_{n+2} + \al_1 + \al_{n+2} \geq 0$ and $x_2 + \ld + x_{n+2} + \al_2 + \ld + \al_{n+2} \geq 0$; 
$x_1=0$ or $x_1+x_{n+2}+\al_1+\al_{n+2}<1$; 
$x_2 + \ld + x_{n+2} + \al_2 +\ld + \al_{n+2}<1$ or $x_i=0$ for all $i \in \{2, \ld, n+1\}$; 
$x_1+x_{n+2}+\al_1+\al_{n+2}<1$ or $x_2 + \ld + x_{n+2} + \al_2 +\ld + \al_{n+2}<1$; 
and $x_1=0$ or $x_i=0$ for all $i \in \{2, \ld, n+1\}$. 

If $x_1=\ld=x_{n+1}=0$, then $x_{n+2}+\al_1+\al_{n+2} \geq 0$ so $x_{n+2} \geq -1$. 
Since we either have $x_{n+2}+\al_1+\al_{n+2}<1$ or $x_{n+2} + \al_2 +\ld + \al_{n+2}<1$, we also have $x_{n+2} \leq 0$.
$x_{n+2}=-1$ gives a apexpoint if and only if $\al_1 + \al_{n+2} \geq 1$ and $\al_2 +\ld + \al_{n+2} \geq1$, and $x_{n+2}=0$ gives an apexpoint in all other cases.
Hence there is always exactly one apexpoint with $x_1=\ld=x_{n+1}=0$. 

If $x_1=0$ and there exists $2 \leq i \leq n+1$ with $x_i>0$, then $x_2 + \ld + x_{n+2} + \al_2 +\ld + \al_{n+2}<1$. 
From $x_{n+2} + \al_1 + \al_{n+2} \geq 0$ it follows that $x_{n+2} \geq -1$.
Now $x_2, \ld, x_{n+1} \geq 0$ and $x_2 + \ld + x_{n+2} \leq 0$ imply that $x_i=1$, $x_{n+2}=-1$ and $x_j=0$ for all $j \neq i, n+2$.
Hence there is at most one apexpoint, and this is indeed an apexpoint if and only if $\al_1+\al_{n+2} \geq 1$ and $\al_2 + \ld + \al_{n+2}<1$. 
Since this condition is independent of $i$, there are either 0 or $n$ apexpoints of this form.

Finally, if $x_1>0$, then we have $0 \leq x_1+x_{n+2}+\al_1+\al_{n+2}<1$, $x_2= \ld = x_{n+1}=0$ and $x_{n+2} + \al_2 + \ld + \al_{n+2} \geq 0$. 
It follows that $-n-1 < -(\al_2 + \ld + \al_{n+2}) \leq x_{n+2} < -(\al_1+\al_{n+2}) \leq 0$, so $-n \leq x_{n+2} \leq -1$.
For every such $x_{n+2}$, there is exactly one $x_1$ which satisfies the conditions, namely $x_1=-x_{n+2}-\entier{\al_1+\al_{n+2}}$.
This gives at most $n$ apexpoints, and there are $n$ apexpoints if and only if $\al_1 + \al_{n+2}<1$ and $\al_2 + \ld + \al_{n+2} \geq n$.

Hence there are $n+1$ apexpoints if and only if either $\al_1+\al_{n+2} \geq 1$ and $\al_2 + \ld + \al_{n+2}<1$, or  $\al_1 + \al_{n+2}<1$ and $\al_2 + \ld + \al_{n+2} \geq n$.
Since $\bal=(1-\fp{a}, 1-\fp{b_1}, \ld, 1-\fp{b_n},\fp{c-1})$, this is equivalent to the condition that $\fp{c} < \fp{a}$ and $\fp{b_1} + \ld + \fp{b_n} \leq \fp{c}$, or $\fp{a} \leq \fp{c}$ and $\fp{c}+n-1 < \fp{b_1} + \ld + \fp{b_n}$.
\end{proof}

Now we have found an interlacing condition, we can easily check whether a parameter vector $(a,b_1,\ld,b_n,c)$ gives rise to an irreducible algebraic function. 
To find all irreducible algebraic functions, we use a reduction to Lauricella functions with less variables:

\begin{lem}\label{lem:FD_reduction}
If $F_D(a,\vect{b},c | \vect{z})$ is an irreducible algebraic function, then for every $i \in \{1, \ld, n\}$, 
$F_D(a, b_1, \ld, b_{i-1}, b_{i+1}, \ld, b_n,c | z_1, \ld, z_{i-1}, z_{i+1}, \ld, z_n)$ is also irreducible and algebraic.
\end{lem}

\begin{proof}
To simplify notation, assume that $i=n$.
Write $F_{D,n-1} = F_D(a, b_1, \ld, b_{n-1},c | z_1, \ld, z_{n-1})$.
From the irreducibility conditions for $F_D(a,\vect{b},c | \vect{z})$, it follows that $F_{D,n-1}$ is irreducible unless $\fp{c}-\fp{b_1}-\ld-\fp{b_{n-1}}$ is an integer.
However, if this is an integer, then $F_D(a,\vect{b},c | \vect{z})$ doesn't satisfy the interlacing condition.
Hence $F_{D,n-1}$ is also irreducible. 
Algebraicity of $F_{D,n-1}$ follows from Remark~\ref{rmk:reduction_algebraicity}.
\end{proof}

\begin{lem}\label{lem:F1_b+c}
If $F_1(a,b_1,b_2,c|x,y)$ is irreducible and algebraic, then $F(a,b_1+b_2,c|z)$ is also irreducible and algebraic.
\end{lem}

\begin{proof}
Since $F_1(a,b_1,b_2,c|x,y)$ is irreducible, $a, c-a$ and $c-b_1-b_2$ are non-integral.
It follows from the interlacing condition for $F_1(a,b_1,b_2,c|x,y)$ that either $0 < \fp{b_1}+\fp{b_2} \leq \fp{c} <1$ or $1<\fp{c}+1<\fp{b_1}+\fp{b_2}<2$, so $b_1+b_2$ cannot be an integer.
Hence $F(a,b_1+b_2,c|z)$ is irreducible.

To prove that $F(a,b_1+b_2,c|z)$ is algebraic, let $D$ be the smallest common denominator of $a, b_1+b_2$ and $c$ and let $\tilde{D}$ be the smallest common denominator of $a,b_1,b_2$ and $c$.
Then clearly $D | \tilde{D}$. 
Let $k \in \Z$ such that $1 \leq k <D$ and $\gcd(k,D)=1$.
By Lemma~\ref{lem:higher_denom}, there exists $l \in \Z$ such that $l \equiv k \pmod D$ and $\gcd(l,\tilde{D})=1$.
Since $F_1(a,b_1,b_2,c|x,y)$ is algebraic, either $\fp{lc} < \fp{la}$ and $\fp{lb_1} + \fp{lb_2} \leq \fp{lc}$, or  $\fp{la} \leq \fp{lc}$ and $\fp{lc}+1 < \fp{lb_1} + \fp{lb_2}$.
If the first condition holds, then $\fp{kc}=\fp{lc} < \fp{la}=\fp{ka}$ and $\fp{k(b_1+b_2)} = \fp{l(b_1+b_2)} \leq \fp{lb_1} + \fp{lb_2} \leq \fp{lc} = \fp{kc}$.
If the second condition holds, then $\fp{ka} \leq \fp{kc}$ and $\fp{kc} < \fp{lb_1} + \fp{lb_2} -1 \leq \fp{l(b_1+b_2)} = \fp{k(b_1+b_2)}$.
In both cases, the interlacing condition for the Gauss functon is satisfied, so $F(a,b_1+b_2,c|z)$ is algebraic.
\end{proof}

\begin{rmk}\label{rmk:F1_b+c=gauss}
One can show that $F(a,b_1+b_2,c|z) = F_1(a,b_1,b_2,c|z,z)$.
Hence algebraicity of $F(a,b_1+b_2,c|z)$ also follows from a statement similar to Remark~\ref{rmk:reduction_algebraicity}.
\end{rmk}

\begin{thm}\label{thm:F1_solutions}
$F_1(a,b_1,b_2,c | x,y)$ is irreducible and algebraic if and only if $(a,b_1,b_2,c) \pmod \Z$ is equals $\pm(\frac{1}{3}, \frac{5}{6}, \frac{5}{6}, \half)$, $\pm(\frac{1}{6}, \frac{2}{3}, \frac{5}{6}, \frac{1}{3})$, $\pm(\frac{1}{6}, \frac{5}{6}, \frac{2}{3}, \frac{1}{3})$, $\pm(\frac{1}{6}, \frac{5}{6}, \frac{5}{6}, \half)$ or $\pm(\frac{1}{6}, \frac{5}{6}, \frac{5}{6}, \frac{1}{3})$. 
\end{thm}

\begin{proof}
If $F_1(a,b_1,b_2,c | x,y)$ is irreducible and algebraic, then by Lemma~\ref{lem:FD_reduction} and Lemma~\ref{lem:F1_b+c}, $(a,b_1,c), (a,b_2,c)$ and $(a,b_1+b_2,c)$ are Gauss triples. 

First suppose that $(a,b_1,c)$ and $(a,b_2,c)$ are both Gauss triples of type 1.
Then there exist $r,s \in \parset$ such that we have $(a,b_1,c) \in \{(r,-r,\half), (r,r+\half,\half), (r,r+\half,2r)\}$ and $(a,b_2,c) \in \{(s,-s,\half), (s,s+\half,\half), (s,s+\half,2s)\}$ (up to congruence modulo $\Z$).
This implies that $r=s$.
This gives five possibilities for $(a,b_1,b_2,c)$ and we obtain the four combinations
\begin{equation*}
(a,b_1+b_2,c) \equiv (r,-2r,\half), (r,\half,\half), (r,2r,\half),(r,2r,2r) \pmod \Z
\end{equation*}
$(r,-2r,\half)$ is of type 1 if and only if $r=\pm \frac{1}{6}$. 
This gives the tuples $(a,b_1,b_2,c)=(\frac{1}{6}, \frac{5}{6}, \frac{5}{6}, \half)$ and $(\frac{5}{6}, \frac{1}{6}, \frac{1}{6}, \half)$.
The triple $(r,2r,\half)$ is of type 1 if and only if $r=\pm \frac{1}{3}$, which gives the tuples $(a,b_1,b_2,c)=(\frac{1}{3}, \frac{5}{6}, \frac{5}{6}, \half)$ and $(\frac{1}{3}, \frac{5}{6}, \frac{5}{6}, \half)$. 
In all other cases, $(a,b_1+b_2,c)$ is of type 2, and hence the denominator of $r$ is at most 60.
This gives finitely many possibilities and for each of these possibilities, we let the computer check whether all conjugates $k(a,b_1,b_2,c)$ (with $\gcd(k,D)=1$, where $D$ is the smallest common denominator of $a,b_1,b_2$ and $c$) satisfy the interlacing condition.
It turns out that this gives the same four solutions as we found above.

If $(a,b_1,c)$ is a Gauss triple of type 1 and $(a,b_2,c)$ is of type 2, then the denominator of $a$ is at most 60.
This gives finitely many possibilities for the parameter $r$ in $(a,b_1,c)$.
Again we check these by computer, and we get the solutions $(\frac{1}{6}, \frac{2}{3}, \frac{5}{6}, \frac{1}{3})$ and $(\frac{5}{6}, \frac{1}{3}, \frac{1}{6}, \frac{2}{3})$. 
By symmetry, if $(a,b_1,c)$ is of type 2 and $(a,b_2,c)$ is of type 1, the solutions are $(\frac{1}{6}, \frac{5}{6}, \frac{2}{3}, \frac{1}{3})$ and $(\frac{5}{6}, \frac{1}{6}, \frac{1}{3}, \frac{2}{3})$. 

Finally, if both $(a,b_1,c)$ and $(a,b_2,c)$ are of type 2, we have a finite list to check. 
This gives two more solutions: $(\frac{1}{6}, \frac{5}{6}, \frac{5}{6}, \frac{1}{3})$ and $(\frac{5}{6}, \frac{1}{6}, \frac{1}{6}, \frac{2}{3})$.
\end{proof}

\begin{thm}\label{thm:FD_solutions}
For $n=3$, $F_D(a,\vect{b},c|\vect{z})$ is irreducible and algebraic if and only if $(a,\vect{b},c) = (\frac{1}{6}, \frac{5}{6}, \frac{5}{6}, \frac{5}{6}, \frac{1}{3})$ or $(\frac{5}{6}, \frac{1}{6}, \frac{1}{6}, \frac{1}{6}, \frac{2}{3}) \pmod \Z$.
For $n \geq 4$, there are no irreducible algebraic Lauricella $F_D$ functions.
\end{thm}

\begin{proof}
Let $n=3$and let $F_D(a,\vect{b},c|\vect{z})$ be irreducible and algebraic.
Then $F_1(a,b_1,b_2,c|x,y)$, $F_1(a,b_1,b_3,c|x,y)$ and $F_1(a,b_2,b_3,c|x,y)$ are also irreducible and algebraic.
Using Theorem~\ref{thm:F1_solutions}, one easily computes that the only irreducible possibilities for $(a,b_1,b_2,b_3,c)$ are $(a,b_1,b_2,b_3,c) = (\frac{1}{6}, \frac{5}{6}, \frac{5}{6}, \frac{5}{6}, \frac{1}{3})$ and $(\frac{5}{6}, \frac{1}{6}, \frac{1}{6}, \frac{1}{6}, \frac{2}{3})$.
They form an orbit and satisfy the interlacing condition. 

For $n=4$, the only possibilities for $(a,b_1,b_2,b_3,b_4,c)$ are $(\frac{1}{6}, \frac{5}{6}, \frac{5}{6}, \frac{5}{6}, \frac{5}{6}, \frac{1}{3})$ and $(\frac{5}{6}, \frac{1}{6}, \frac{1}{6}, \frac{1}{6}, \frac{1}{6}, \frac{2}{3})$. 
However, both functions are reducible.
Hence there are no irreducible algebraic functions in 4 variables. 
Lemma\ref{lem:FD_reduction} implies that there are also no irreducible algebraic $F_D$ functions in 5 or more variables.
\end{proof}


\subsection{The Appell $F_2$ and Lauricella $F_A$ functions}\label{subsec:F2}

The Lauricella $F_A$ function is defined by 
\begin{equation*}
F_A(a,\vect{b},\vect{c} | \vect{z}) =
\sum_{\vect{m} \in \Z_{\geq 0}^n} \frac{\poch{a}{|\vect{m}|} \poch{\vect{b}}{\vect{m}}}{\poch{\vect{c}}{\vect{m}} \vect{m}!} \vect{z}^{\vect{m}}. 
\end{equation*}
The lattice is $\lat = \bigoplus_{i=1}^n \Z(-\e_1-\e_{i+1}+\e_{n+i+1}+\e_{2n+i+1}) \subs \Z^{3n+1}$
and we can take 
\begin{align*}
\gam & = (-a, -\vect{b}, \vect{c}-1, \vect{0}) \in \R^{3n+1}, \\
\A & = \{\e_1, \e_2, \ld, \e_{2n+1}, \e_1+\e_2-\e_{n+2}, \e_1+\e_3-\e_{n+3}, \ld, \e_1+\e_{n+1}-\e_{2n+1}\} \subs \Z^{2n+1}, \\
\bal & = \sum_{i=1}^{3n+1} \gamma_i \va_i = (-a, -\vect{b}, \vect{c}-1) \in \Q^{2n+1}.
\end{align*}
where $\vect{c}-1= (c_1-1, \ld, c_n-1)$.
For each $I \subs \{n+2, \ld, 2n+1\}$, define $\tilde{I} = \{n+2, \ld, 2n+1\} \setminus I$ and $V_I = \{\e_1, \ld, \e_{n+1}\} \cup \{\e_1+\e_{i-n}-\e_i | i \in I\} \cup \{\e_i | i \in \tilde{I}\}$.
Then the determinant of the vectors in $V_I$ is $\pm 1$, so the vectors are the vertices of a $2n$-dimensional simplex with volume 1.

\begin{lem}\label{lem:FA_CVi}
Let $I \subs \{n+2, \ld, 2n+1\}$. 
Then
\begin{equation*}
\begin{split}
C(V_I) = \{\x \in \R^{2n+1} \ | \ & x_1, \ld, x_{n+1} \geq 0; x_1+\sum_{i \in I} x_i \geq 0; \\ 
 & \forall i \in \tilde{I}: x_i \geq 0; \forall i \in I: x_i \leq 0 \andd x_{i-n}+x_i \geq 0\}.
\end{split}
\end{equation*}
\end{lem}

\begin{proof}
This follows from  
\begin{equation*}
\begin{split}
C(V_I) = \{\x \in \R^{2n+1} \ | \ & \exists \la_1, \ld, \la_{2n+1} \geq 0:
x_1 = \la_1 + \sum_{i \in I} \la_i;
\forall i \in \tilde{I}: x_{i-n}=\la_{i-n} \andd x_i=\la_i; \\
 & \forall i \in I: x_{i-n}=\la_{i-n}+\la_i \andd x_i=-\la_i 
\}.
\qedhere
\end{split}
\end{equation*}
\end{proof}

\begin{cor}\label{cor:FA_CA}
\begin{equation*}
\begin{split}
\bigcup_I C(V_I) = \{\x \in \R^{2n+1} \ | \ & x_1, \ld, x_{n+1} \geq 0;
\textrm{for all } I \subs \{n+2, \ld, 2n+1\}: x_1 + \sum_{i \in I} x_i \geq 0; \\
 & \textrm{for all } i \in \{n+2, \ld, 2n+1\}: x_{i-n}+x_i \geq 0\}.
\end{split}
\end{equation*}
\end{cor}

\begin{lem}\label{lem:FA_trian}
$\trian=\{Q(V_I) \ | \ I \subs \{n+2, \ld, 2n+1\}\}$ is a triangulation of $Q(\A)$.
\end{lem}

\begin{proof}
By Lemma~\ref{lem:convex_hull}, it suffices to prove that $\cup_I C(V_I)$ is convex and $C(V_I) \cap C(V_J) \subs C(V_I \cap V_J)$ for all $I, J \subs \{n+2, \ld, 2n+1\}$.
The first statement follows from Corollary~\ref{cor:FA_CA}.
For the second statement, one can easily show that both $C(V_I) \cap C(V_J)$ and $C(V_I \cap V_J)$ equal 
\begin{equation*}
\begin{split}
\{\x \in \R^{2n+1} \ | \ 
 & x_1, \ld, x_{n+1} \geq 0; 
x_1+\sum_{i \in I \cap J} x_i \geq 0; 
\forall i \in \tilde{I} \cap \tilde{J}: x_i \geq 0; \\
 & \forall i \in I \cap J: x_i \leq 0 \andd x_{i-n}+x_i \geq 0;
\forall i \in (I \cap \tilde{J}) \cup (\tilde{I} \cap J): x_i=0\}.
\qedhere
\end{split}
\end{equation*}
\end{proof}

\begin{cor}\label{cor:FA_prop_sd}
$\A$ is saturated, the volume of $Q(\A)$ is $2^n$ and 
\begin{equation*}
\begin{split}
C(\A) = \{\x \in \R^{2n+1} \ | \ & x_1, \ld, x_{n+1} \geq 0;
\textrm{for all } I \subs \{n+2, \ld, 2n+1\}: x_1 + \sum_{i \in I} x_i \geq 0; \\
 & \textrm{for all } i \in \{n+2, \ld, 2n+1\}: x_{i-n}+x_i \geq 0\}
\end{split}
\end{equation*}
$F_A(a,\vect{b},\vect{c} | \vect{z})$ is non-resonant if and only if $b_1, \ld, b_n$, $c_1-b_1, \ld, c_n-b_n$ and $-a + \sum_{j \in J} c_j$ are non-integral for all $J \subseteq \{1, \ld, n\}$.
\end{cor}

\begin{cor}\label{cor:FA_reduction}
If $F_A(a,\vect{b},\vect{c} | \vect{z})$ is non-resonant and algebraic, then for every $i \in \{1, \ld, n\}$,
$F_A(a, b_1, \ld, b_{i-1}, b_{i+1}, \ld, b_n,c_1, \ld, c_{i-1}, c_{i+1}, \ld, c_n \hfill | \hfill z_1, \ld, z_{i-1}, z_{i+1}, \ld, z_n)$ is non-resonant and algebraic.
\end{cor}

\begin{rmk}\label{rmk:F2_irreducible_is_nonresonant}
In section 2 of \cite{kato_Appell_F2}, Kato shows that $F_2(a,b_1,b_2,c_1,c_2|x,y)$ is irreducible if and only if $a$, $b_1$, $b_2$, $c_1-a$, $c_2-a$, $c_1+c_2-a$, $c_1-b_1$ and $c_2-b_2$ are non-integral.
However, we could not find similar results for $F_A$ in the literature.
\end{rmk}

Finding an interlacing condition isn't as easy as for $F_D$.
Therefore, we will first find an interlacing condition for the Appell function and compute all irreducible algebraic functions for $n=2$.
Using this, we can prove that there are no non-resonant algebraic functions for $n \geq 3$. 

\begin{lem}\label{lem:F2_interlacing}
Suppose that $F_2(a,b_1,b_2,c_1,c_2 | x,y)$ is irreducible.
Then there are 4 apexpoints if and only if one of the following conditions holds:
\begin{equation*}
\fp{b_1} \leq \fp{c_1} \andd \fp{b_2} \leq \fp{c_2} \andd \fp{c_1} + \fp{c_2} < \fp{a} 
\end{equation*}
or
\begin{equation*}
\fp{b_1} \leq \fp{c_1} \andd \fp{c_2} < \fp{b_2} \andd \fp{c_1} < \fp{a} \leq \fp{c_2} 
\end{equation*}
or
\begin{equation*}
\fp{c_1} < \fp{b_1} \andd \fp{b_2} \leq \fp{c_2} \andd \fp{c_2} < \fp{a} \leq \fp{c_1} 
\end{equation*}
or
\begin{equation*}
\fp{c_1} < \fp{b_1} \andd \fp{c_2} < \fp{b_2} \andd 1 + \fp{a} \leq \fp{c_1} + \fp{c_2}
\end{equation*}
\end{lem}

\begin{proof}
Using the algorithm described in Remark~\ref{rmk:algorithm_interlacing} and the assumption $\al_i \in [0,1)$, one easily computes that there are 4 apexpoints if and only if 
$(\entier{\al_1+\al_4}, \entier{\al_1+\al_5}, \entier{\al_1+\al_4+\al_5}, \entier{\al_2+\al_4}, \entier{\al_3+\al_5}) \in \{(1,1,2,0,0), (0,1,1,1,0), (1,0,1,0,1), (0,0,0,1,1)\}$
Since $\bal=(1-\fp{a},1-\fp{b_1},1-\fp{b_2},\fp{c_1},\fp{c_2})$, this is equivalent to the conditions given above.
\end{proof}

\begin{lem}\label{lem:F2_a-c_2}
If $F_2(a,b_1,b_2,c_1,c_2|x,y)$ is irreducible and algebraic, then $F(a-c_2,b_1,c_1|x)$ is also irreducible and algebraic.
\end{lem}

\begin{proof}
The proof is similar to the proof of Lemma~\ref{lem:F1_b+c} and uses the interlacing condition from Lemma~\ref{lem:F2_interlacing}.
\end{proof}

Since $F_2(a,b_1,b_2,c_1,c_2|x,y) = F_2(a,b_2,b_1,c_2,c_1|y,x)$, the algebraic functions come in pairs. 
 
\begin{thm}\label{thm:F2_solutions}
$F_2(a,b_1,b_2,c_1,c_2 | x,y)$ is irreducible and algebraic if and only if $(a,b_1,b_2,c_1,c_2) \pmod \Z$ or $(a,b_2,b_1,c_2,c_1) \pmod \Z$ is conjugate to one of the tuples $(\half, \frac{1}{6}, \frac{5}{6}, \frac{1}{3}, \frac{2}{3})$, $(\frac{1}{6}, \frac{5}{6}, \frac{5}{6}, \frac{2}{3}, \frac{2}{3})$, $(\frac{1}{10}, \frac{7}{10}, \frac{9}{10}, \frac{2}{5}, \frac{4}{5})$, $(\frac{1}{12}, \frac{3}{4}, \frac{5}{6}, \half, \frac{2}{3})$, $(\frac{1}{12}, \frac{5}{6}, \frac{11}{12}, \frac{2}{3}, \half)$, $(\frac{1}{12}, \frac{5}{6}, \frac{7}{12}, \frac{2}{3}, \half)$ and $(\frac{1}{30}, \frac{5}{6}, \frac{7}{10}, \frac{2}{3}, \frac{2}{5})$. 
\end{thm}

\begin{proof}
If $F_2(a,b_1,b_2,c_1,c_2 | x,y)$ is irreducible and algebraic, then Corollary~\ref{cor:FA_reduction} and Lemma~\ref{lem:F2_a-c_2} imply that $(a,b_1,c_1)$, $(a,b_2,c_2)$ and $(a-c_2,b_1,c_1)$ are Gauss triples.

First suppose that $(a,b_1,c_1)$ and $(a,b_2,c_2)$ are both Gauss triples of type 1.
Then there exist $r \in \parset$ such that $(a,b_1,c_1), (a,b_2,c_2) \in \{(r,-r,\half), (r,r+\half,\half), (r,r+\half,2r)\}$ (up to congruence modulo $\Z$).
Hence $a-c_2, b_1 \in \{r+\half, -r\} \pmod \Z$.
If $(a-c_2,b_1,c_1)$ is a Gauss triple of type 1, then $a-c_2 \equiv -b_1$ or $a-c_2 \equiv b_1+\half$.
However, this doesn't hold for $r \neq \half$.
Hence $(a-c_2, b_1, c_1)$ must be of type 2, so the denominator of $a-c_2$ is at most 60.
This implies that the denominator of $r$ is at most 60, or $2 \pmod 4$ and at most 120.
This gives finitely many possibilities for $r$ and using a computer it turns out that there are no solutions.

If $(a,b_1,c_1)$ is a Gauss triple of type 1 and $(a,b_2,c_2)$ is of type 2, then the denominator of $a$ is at most 60 and there are again finitely many possibilities.
The solutions are the 8 points in the orbits of $(\frac{1}{12}, \frac{11}{12}, \frac{5}{6}, \half, \frac{2}{3})$ and $(\frac{1}{12}, \frac{7}{12}, \frac{5}{6}, \half, \frac{2}{3})$.
By symmetry, if $(a,b_1,c_1)$ is of type 2 and $(a,b_2,c_2)$ is of type 1, the solutions are the conjugates of $(\frac{1}{12}, \frac{5}{6}, \frac{11}{12}, \frac{2}{3}, \half)$ and $(\frac{1}{12}, \frac{5}{6}, \frac{7}{12}, \frac{2}{3}, \half)$.

Finally, when both $(a,b_1,c_1)$ and $(a,b_2,c_2)$ are of type 2, there are only finitely many possibilities.
This gives the other 36 solutions listed above.
\end{proof}

\begin{thm}\label{thm:FA_solutions}
For $n \geq 3$, there are no non-resonant algebraic Lauricella $F_A$ functions.
\end{thm}

\begin{proof}
First let $n=3$.
If $F_A(a,\vect{b},\vect{c}|\vect{z})$ is non-resonant and algebraic, then each of the three tuples $(a,b_1,b_2,c_1,c_2)$, $(a,b_1,b_3,c_1,c_3)$ and $(a,b_2,b_3,c_2,c_3)$ must give an irreducible algebraic $F_2$ function.
From Theorem~\ref{thm:F2_solutions}, it easily follows that $(a,\vect{b},\vect{c})= \pm (\frac{1}{6}, \frac{5}{6}, \frac{5}{6}, \frac{5}{6}, \frac{2}{3}, \frac{2}{3}, \frac{2}{3})$.
Hence $\bal$ is equal to $\pm (\frac{1}{6}, \frac{5}{6}, \frac{5}{6}, \frac{5}{6}, \frac{1}{3}, \frac{1}{3}, \frac{1}{3})$ and the corresponding functions are non-resonant.
There are 5 and 7 apexpoints, respectively.
Since the volume of $Q(\A)$ is 8, the functions are not algebraic.
By Corollary~\ref{cor:FA_reduction}, this implies that there are no non-resonant algebraic functions for $n \geq 4$.
\end{proof}


\subsection{The Appell $F_3$ and Lauricella $F_B$ functions}\label{subsec:F3}

The Lauricella $F_B$ function is defined by 
\begin{equation*}
F_B(\vect{a},\vect{b},c | \vect{z}) = 
\sum_{\vect{m} \in \Z_{\geq 0}^n} \frac{\poch{\vect{a}}{\vect{m}} \poch{\vect{b}}{\vect{m}}}{\poch{c}{|\vect{m}|} \vect{m}!} \vect{z}^{\vect{m}}.
\end{equation*}
The lattice is $\lat = \bigoplus_{i=1}^n \Z(-\e_i-\e_{n+i}+\e_{2n+1}+\e_{2n+i+1}) \subs \Z^{3n+1}$.
We can take
\begin{equation*}
\A = \{\e_1, \e_2, \ld, \e_{2n+1}, \e_1+\e_2-\e_{n+2}, \e_1+\e_3-\e_{n+3}, \ld, \e_1+\e_{n+1}-\e_{2n+1}\} \subs \Z^{2n+1}
\end{equation*}
and $\gam = (-\vect{a}, -\vect{b}, c-1, \vect{0})$.
Then $\bal = (-\vect{a}, -\vect{b}, c-1)$.
Consider the map $f:\Z^{2n+1} \rightarrow \Z^{2n+1}: x \mapsto (x_2+x_{n+2}, \ld, x_{n+1}+x_{2n+1}, x_2, \ld, x_{n+1}, x_1-x_2-\ld-x_{n+1}$.
Its inverse is $f^{-1}: x \mapsto (x_{n+1}+\ld+x_{2n}+x_{2n+1}, x_{n+1}, \ld, x_{2n}, x_1-x_{n+1}, \ld, x_n+x_{2n})$.
It maps the set $\A$ of $F_A$ to the set $\A$ of $F_B$.
Hence Remark~\ref{rmk:transformations} gives the following results:

\begin{lem}\label{lem:FB_irreducible}
$F_B(\vect{a},\vect{b},c | \vect{z})$ is non-resonant if and only if $a_1, \ld, a_n$, $b_1, \ld, b_n$ and $c-d_1- \ld -d_n$ with $d_i \in \{a_i,b_i\}$ are non-integral.
\end{lem}

\begin{thm}\label{thm:FB_solutions}
$F_3(a_1,a_2,b_1,b_2,c | x,y)$ is non-resonant and algebraic if and only if, up to equivalence modulo $\Z$, $(a_1,a_2,b_1,b_2,c), (a_2,a_1,b_2,b_1,c), (b_1,b_2,a_1,a_2,c)$ or $(b_2,b_1,a_2,a_1,c)$ is conjugate to 
$(\frac{1}{4}, \frac{1}{6}, \frac{3}{4}, \frac{5}{6}, \half)$, 
$(\frac{1}{6}, \frac{1}{6}, \frac{5}{6}, \frac{5}{6}, \half)$, 
$(\frac{1}{6}, \frac{1}{10}, \frac{5}{6}, \frac{9}{10}, \half)$, 
$(\frac{1}{6}, \frac{1}{12}, \frac{5}{6}, \frac{7}{12}, \frac{1}{3})$ or 
$(\frac{1}{10}, \frac{3}{10}, \frac{9}{10}, \frac{7}{10}, \frac{1}{2})$.
There are no non-resonant algebraic Lauricella $F_B$ functions for $n \geq 3$.
\end{thm}


\subsection{The Appell $F_4$ and Lauricella $F_C$ functions}\label{subsec:F4}

The Lauricella $F_C$ function is defined by 
\begin{equation*}
F_C(a,b,\vect{c} | \vect{z}) = 
\sum_{\vect{m} \in \Z_{\geq 0}^n} \frac{\poch{a}{|\vect{m}|} \poch{b}{|\vect{m}|}}{\poch{\vect{c}}{\vect{m}} \vect{m}!} \vect{z}^{\vect{m}}.  
\end{equation*}
The lattice is $\lat = \bigoplus_{i=1}^n \Z(-\e_1-\e_2+\e_{i+2}+\e_{n+i+2}) \subs \Z^{2n+2}$ and we can choose 
\begin{equation*}
\A = \{\e_1, \e_2, \ld, \e_{n+2}, \e_1+\e_2-\e_3, \e_1+\e_2-\e_4, \ld, \e_1+\e_2-\e_{n+2}\} \subs \Z^{n+2}.
\end{equation*}
We have $\gam = (-a, -b, \vect{c}-1, \vect{0})$, so $\bal = (-a, -b, \vect{c}-1) \in \Q^{n+2}$. \\

For $I \subs \{3, \ld, n+2\}$, let $\tilde{I} = \{3, \ld, n+2\} \setminus I$ and $V_I = \{\e_1,\e_2\} \cup \{\e_i | i \in \tilde{I}\} \cup \{\e_1+\e_2-\e_i | i \in I\}$.
The determinant of the vectors in $V_I$ equals $\pm 1$, so the vectors are the vertices of an $(n+1)$-dimensional simplex.

\begin{lem}\label{lem:FC_CVi}
For $I \subs \{3, \ld, n+2\}$, we have
\begin{equation*}
C(V_I) = \{\x \in \R^{n+2} \ | \ x_1, x_2 \geq 0; \forall i \in I: x_i \leq 0; \forall i \in \tilde{I}: x_i \geq 0; x_1+\sum_{i \in I} x_i \geq 0; x_2+\sum_{i \in I} x_i \geq 0\}.
\end{equation*}
\end{lem}

\begin{cor}\label{cor:FC_CA}
\begin{equation*}
\bigcup_I C(V_I) = \{\x \in \R^{n+2} \ | \ \forall I \subs \{3, \ld, n+2\}: x_1+\sum_{i \in I} x_i \geq 0; x_2+\sum_{i \in I} x_i \geq 0\}.
\end{equation*}
\end{cor}

\begin{lem}\label{lem:FC_trian}
$\trian = \{Q(V_I) \ | \ I \subs \{3, \ld, n+2\}\}$ is a triangulation of $Q(\A)$.
\end{lem}

\begin{proof}
By Lemma~\ref{lem:convex_hull}, it suffices to prove that $\cup_I C(V_I)$ is convex and $C(V_I) \cap C(V_J) \subs C(V_I \cap V_J)$ for all $I, J \subs \{3, \ld, n+2\}$.
The first statement follows from Corollary~\ref{cor:FC_CA}.
For the second statement, one can show that both $C(V_I) \cap C(V_J)$ and $C(V_I \cap V_J)$ equal 
\begin{equation*}
\begin{split}
\{\x \in \R^{n+2} \ | \ & x_1, x_2 \geq 0; 
\forall i \in I \cap J: x_i \leq 0; 
\forall i \in \tilde{I} \cap \tilde{J}: x_i \geq 0; \\
 & \forall i \in (I \cap \tilde{J}) \cup (\tilde{I} \cap J): x_i = 0;
x_1+\sum_{i \in I \cap J} x_i \geq 0; 
x_2+\sum_{i \in I \cap J} x_i \geq 0\}.
\qedhere
\end{split}
\end{equation*}
\end{proof}

\begin{cor}\label{cor:FC_prop_trian}
$\A$ is saturated, the volume of $Q(\A)$ is $2^n$ and 
\begin{equation*}
C(\A) = \{\x \in \R^{n+2} \ | \ \forall I \subs \{3, \ld, n+2\}: x_1+\sum_{i \in I} x_i \geq 0; x_2+\sum_{i \in I} x_i \geq 0\}.
\end{equation*}
$F_C(a,b,\vect{c} | \vect{z})$ is non-resonant if and only if $-a + \sum_{i \in I} c_i$ and $-b + \sum_{i \in I} c_i$ are non-integral for all $I \subs \{1, \ld, n\}$.
\end{cor}

\begin{cor}\label{cor:FC_reduction}
If $F_C(a,b,\vect{c} | \vect{z})$ is non-resonant and algebraic, then for every $i \in \{1, \ld, n\}$, 
$F_C(a, b, c_1, \ld, c_{i-1}, c_{i+1}, \ld, c_n | z_1, \ld, z_{i-1}, z_{i+1}, \ld, z_n)$ is also non-resonant and algebraic.
\end{cor}

\begin{rmk}\label{rmk:F4_irreducible_is_nonresonant}
Kato shows that $F_4(a,b,c_1,c_2|x,y)$ is irreducible if and only if $a$, $c_1-a$, $c_2-a$, $c_1+c_2-a$, $b$, $c_1-b$, $c_2-b$ and $c_1+c_2-b$ are non-integral (see~\cite[Theorem~7.2]{kato_Appell_F4_connection_formulas} and~\cite[Theorem~1]{kato_Appell_F4_irreducibilities}).
However, we could not find similar results for $F_C$ in the literature.
\end{rmk}

\begin{lem}\label{lem:F4_interlacing}
Let $\fp{a} \leq \fp{b}$.
Suppose that $F_4(a,b,c_1,c_2|x,y)$ is irreducible.
Then there are 4 apexpoints if and only if  
$\fp{a} \leq \fp{c_1}, \fp{c_2} < \fp{b} \leq \fp{c_1}+\fp{c_2} < \fp{a}+1$.
If $F_C(a,b,c_1,c_2,c_3|z_1,z_2,z_3)$ is non-resonant, then there are 8 apexpoints if and only if 
$\fp{a} \leq \fp{c_1}, \fp{c_2}, \fp{c_3} < \fp{b} \leq \fp{c_1}+\fp{c_2}, \fp{c_1}+\fp{c_3}, \fp{c_2}+\fp{c_3} < \fp{a}+1 \leq \fp{c_1}+\fp{c_2}+\fp{c_3} < \fp{b}+1$.
\end{lem}

\begin{proof}
For $F_4(a,b,c_1,c_2|x,y)$, one easily computes that there are 4 apexpoints if and only if 
$(\entier{\al_1+\al_3}, \entier{\al_1+\al_4}, \entier{\al_1+\al_3+\al_4}, \entier{\al_2+\al+3}, \entier{\al_2+\al_4}, \entier{\al_2+\al_3+\al_4})$ equals $(1,1,1,0,0,1)$ or $(0,0,1,1,1,1)$.
Note that $F_C(a,b,c_1,c_2,c_3|z_1,z_2,z_3)$ can only be algebraic if all induced $F_4$ functions are also algebraic. 
Hence we only have to find the values of $(\entier{\al_1+\al_3+\al_4+\al_5}, \entier{\al_2+\al_3+\al_4+\al_5})$, given the values of the other linear forms, as induced by the $F_4$ functions.
This can again easily be done by the algorithm of Remark~\ref{rmk:algorithm_interlacing}.
\end{proof}

\begin{lem}\label{lem:F4_a+b=c1+c2}
Suppose that $F(a,b,c_1|x)$ and $F(a,b,c_2|x)$ are irreducible and algebraic, and either $a+b \equiv c_1+c_2 \pmod \Z$, or at least two of $c_1, c_2$ and $b-a$ are equivalent to $\half$ modulo $\Z$.
Then $F_4(a,b,c_1,c_2|x,y)$ is also irreducible and algebraic.
\end{lem}

\begin{proof}
Note that $F_4(a,b,c_1,c_2|x,y)$ is irreducible if $c_1+c_2-a$ and $c_1+c_2-b$ are non-integral.

Suppose that $a+b \equiv c_1+c_2 \pmod \Z$.
Then $c_1+c_2-a \equiv b$ is non-integral, and the same holds for $c_1+c_2-b \equiv a$.
Let $k$ be coprime with the denominators of $a,b,c_1$ and $c_2$.
Then we can assume that $\fp{ka} \leq \fp{kc_1}, \fp{kc_2} < \fp{kb}$.
Then $\fp{kc_1}+\fp{kc_2}=\fp{ka}+\fp{kb}$, so the interlacing condition is satisfied.

Now suppose that at least two of $c_1, c_2$ and $b-a$ are equivalent to $\half$ modulo $\Z$.
We can assume that $c_1 \equiv \half$ and $\fp{ka} \leq \half, \fp{kc_2} < \fp{kb}$.
If $c_2 \equiv \half$, then $c_1+c_2-a \equiv -a$ and $c_1+c_2-b\equiv -b$ are non-integral and the interlacing condition is satisfied.
If $b-a \equiv \half$, then $c_1+c_2-a \equiv c_2-b$ and $c_1+c_2-b \equiv c_2-a$ are non-integral.
Since $\fp{ka} \leq \half, \fp{kc_2} < \fp{ka} + \half \leq 1+\left( \fp{kc_2}-\half \right) < 1+\fp{ka}$, the interlacing condition is again satisfied.
\end{proof}

\begin{thm}\label{thm:F4_solutions}
$F_4(a,b,c_1,c_2 | x,y)$ is irreducible and algebraic if and only if $(a,b,c_1)$ and $(a,b,c_2)$ are Gauss triples, and either $a+b \equiv c_1+c_2 \pmod \Z$, or at least two of $c_1, c_2$ and $b-a$ are equivalent to $\half$ modulo $\Z$.
Up to conjugation and permutations of $\{a,b\}$ and of $\{c_1,c_2\}$, the parameters of the irreducible algebraic functions are the tuples in Table~\ref{tab:F4_solutions}.
\end{thm}

\begin{proof}
It suffices to find all tuples satisfying the interlacing condition and prove that they satisfy $a+b \equiv c_1+c_2 \pmod \Z$ or at least two of $c_1, c_2$ and $b-a$ are equivalent to $\half$ modulo $\Z$.

First suppose that $(a,b,c_1)$ and $(a,b,c_2)$ are both Gauss triples of type 1.
Then we have $(a,b,c_1,c_2) \in \{(r,-r,\half,\half), (r,r+\half,\half,\half), (r,r+\half,\half,2r), (r,r+\half,2r,\half), (r,r+\half,2r,2r)\}$ for some $r \in \parset$ (up to equivalence modulo $\Z$).
By Lemma~\ref{lem:F4_a+b=c1+c2}, all these tuples give algebraic functions, possibly except for $(r,r+\half,2r,2r)$.
So suppose that $(a,b,c_1,c_2)=(r,r+\half,2r,2r)$.
Write $r=\frac{p}{q}$ with $\gcd(p,q)=1$.
Then for every $k$ coprime with $2q$ such that $\fp{kr}>\half$, the interlacing condition implies that $\fp{kr}-\half \leq 2\fp{kr}-1 < \fp{kr} \leq 4\fp{kr}-2 < \fp{kr}+\half$, so $\fp{kr} \in [\frac{2}{3}, \frac{5}{6})$.
Hence for every $k$ with $\gcd(k,2q)=1$, it must hold that $\fp{kr} < \frac{5}{6}$.
There exists $k$ such that $kp \equiv -1 \pmod q$.
Choose $l$ such that $\gcd(l,2q)=1$ and $lp \equiv -1 \pmod q$.
If $q \geq 6$, then $\fp{lr} = \frac{q-1}{q} \geq \frac{5}{6}$.
Contradiction, so for all algebraic solutions, it holds that $q<6$. 
By symmetry, $r+\half$ also has a denominator smaller than 6, so $r=\frac{1}{4}$ or $\frac{3}{4}$.
However, this only gives solutions of the form $(r,-r,\half,\half)$. 

If $(a,b,c_1)$ is a Gauss triple of type 1, and $(a,b,c_2)$ is a Gauss triple of type 2, then the denominator of $a$ is at most 60. 
This gives 72 solutions, which all turn out to satisfy $c_1=\half$ and $b-a \equiv \half \pmod \Z$. 
By symmetry, if $(a,b,c_1)$ is of type 2 and $(a,b,c_2)$ is of type 1, we get 72 solutions which all satisfy $c_2=\half$ and $b-a \equiv \half \pmod \Z$. 

Finally, if $(a,b,c_1)$ and $(a,b,c_2)$ are both of type 1, then there are 480 irreducible algebraic functions, and all the tuples $(a,b,c_1,c_2)$ either satisfy $a+b \equiv c_1+c_2 \pmod \Z$, or at least two of $c_1, c_2$ and $b-a$ are equivalent to $\half$ modulo $\Z$.
\end{proof}

\renewcommand{\arraystretch}{2}
\begin{table}[htb]
\begin{center}
\caption{The tuples $(a,b,c_1,c_2)$ such that $F_4(a,b,c_1,c_2|x,y)$ is irreducible and algebraic}
\label{tab:F4_solutions} 
\begin{tabularx}{\textwidth}{XXXXXX} 
\hline
$(r,		-r,		\frac{1}{2},	\frac{1}{2})$ &
$(r,		r+\frac{1}{2},	\frac{1}{2},	\frac{1}{2})$ &
\multicolumn{2}{l}{$(r,		r+\frac{1}{2},	\frac{1}{2},	2r)$} & 
\multicolumn{2}{l}{with $r \in \parset$} \\

$(\frac{1}{2}, 	\frac{1}{6}, 	\frac{1}{3}, 	\frac{1}{3})$ &
$(\frac{1}{4}, 	\frac{3}{4}, 	\frac{1}{2}, 	\frac{1}{3})$ &
$(\frac{1}{4}, 	\frac{3}{4}, 	\frac{1}{3}, 	\frac{2}{3})$ &
$(\frac{1}{4}, 	\frac{7}{12}, 	\frac{1}{2}, 	\frac{1}{2})$ &
$(\frac{1}{4}, 	\frac{7}{12}, 	\frac{1}{2}, 	\frac{1}{3})$ &
$(\frac{1}{6}, 	\frac{5}{6}, 	\frac{1}{3}, 	\frac{2}{3})$ \\

$(\frac{1}{6},	\frac{5}{6}, 	\frac{1}{4}, 	\frac{3}{4})$ &
$(\frac{1}{6}, 	\frac{5}{6}, 	\frac{1}{5}, 	\frac{4}{5})$ &
$(\frac{1}{6}, 	\frac{5}{12}, 	\frac{1}{3}, 	\frac{1}{4})$ &
$(\frac{1}{6}, 	\frac{11}{30}, 	\frac{1}{3}, 	\frac{1}{5})$ &
$(\frac{1}{10}, \frac{3}{10}, 	\frac{1}{5}, 	\frac{1}{5})$ &
$(\frac{1}{10}, \frac{7}{10}, 	\frac{2}{5}, 	\frac{2}{5})$ \\

$(\frac{1}{10}, \frac{9}{10}, 	\frac{1}{3}, 	\frac{2}{3})$ &
$(\frac{1}{10}, \frac{9}{10}, 	\frac{1}{5}, 	\frac{4}{5})$ &
$(\frac{1}{10}, \frac{13}{30}, 	\frac{1}{3}, 	\frac{1}{5})$ &
$(\frac{1}{12}, \frac{5}{12}, 	\frac{1}{4}, 	\frac{1}{4})$ &
$(\frac{1}{12}, \frac{7}{12}, 	\frac{1}{2}, 	\frac{1}{3})$ &
$(\frac{1}{12}, \frac{7}{12}, 	\frac{1}{3}, 	\frac{1}{3})$ \\

$(\frac{1}{15}, \frac{7}{15}, 	\frac{1}{3}, 	\frac{1}{5})$ &
$(\frac{1}{15}, \frac{11}{15}, 	\frac{1}{5}, 	\frac{3}{5})$ &
$(\frac{1}{15}, \frac{13}{15}, 	\frac{1}{3}, 	\frac{3}{5})$ &
$(\frac{1}{20}, \frac{11}{20}, 	\frac{1}{2}, 	\frac{1}{5})$ &
$(\frac{1}{20}, \frac{11}{20}, 	\frac{1}{2}, 	\frac{2}{5})$ &
$(\frac{1}{20}, \frac{11}{20}, 	\frac{1}{5}, 	\frac{2}{5})$ \\

$(\frac{1}{20}, \frac{13}{20}, 	\frac{1}{2}, 	\frac{1}{2})$ &
$(\frac{1}{20}, \frac{13}{20}, 	\frac{1}{2}, 	\frac{1}{5})$ &
$(\frac{1}{20}, \frac{17}{20}, 	\frac{1}{2}, 	\frac{1}{2})$ &
$(\frac{1}{20}, \frac{17}{20}, 	\frac{1}{2}, 	\frac{2}{5})$ &
$(\frac{1}{24}, \frac{13}{24}, 	\frac{1}{2}, 	\frac{1}{3})$ &
$(\frac{1}{24}, \frac{13}{24}, 	\frac{1}{2}, 	\frac{1}{4})$ \\

$(\frac{1}{24}, \frac{13}{24}, 	\frac{1}{3}, 	\frac{1}{4})$ &
$(\frac{1}{24}, \frac{17}{24}, 	\frac{1}{2}, 	\frac{1}{2})$ &
$(\frac{1}{24}, \frac{17}{24}, 	\frac{1}{2}, 	\frac{1}{4})$ &
$(\frac{1}{24}, \frac{19}{24}, 	\frac{1}{2}, 	\frac{1}{2})$ &
$(\frac{1}{24}, \frac{19}{24}, 	\frac{1}{2}, 	\frac{1}{3})$ &
$(\frac{1}{30}, \frac{11}{30}, 	\frac{1}{5}, 	\frac{1}{5})$ \\

$(\frac{1}{30}, \frac{19}{30}, 	\frac{1}{3}, 	\frac{1}{3})$ &
$(\frac{1}{60}, \frac{31}{60}, 	\frac{1}{2}, 	\frac{1}{3})$ &
$(\frac{1}{60}, \frac{31}{60}, 	\frac{1}{2}, 	\frac{1}{5})$ &
$(\frac{1}{60}, \frac{31}{60}, 	\frac{1}{3}, 	\frac{1}{5})$ &
$(\frac{1}{60}, \frac{41}{60}, 	\frac{1}{2}, 	\frac{1}{2})$ &
$(\frac{1}{60}, \frac{41}{60}, 	\frac{1}{2}, 	\frac{1}{5})$ \\

$(\frac{1}{60}, \frac{49}{60}, 	\frac{1}{2}, 	\frac{1}{2})$ &
$(\frac{1}{60},	\frac{49}{60}, 	\frac{1}{2}, 	\frac{1}{3})$ \\
\hline 
\end{tabularx}
\end{center}
\end{table}
\renewcommand{\arraystretch}{1}

\begin{rmk}\label{rmk:Kato_F4}
We proved Theorem~\ref{thm:F4_solutions} by computing all tuples for which $F_4(a,b,c_1,c_2 | x,y)$ is irreducible and algebraic, and checking whether for each of the tuples either $a+b \equiv c_1+c_2 \pmod \Z$, or at least two of $c_1, c_2$ and $b-a$ are equivalent to $\half$ modulo $\Z$.
Unfortunately, this doesn't give much insight.
In \cite{kato_Appell_F4}, Kato proves the same Theorem using monodromy groups, without computing all solutions explicitly.
However, to find all non-resonant algebraic $F_C$ functions in more than 2 variables, we do need to know the solutions for $n=2$ explicitly.
\end{rmk}

For the Lauricella $F_D, F_A$ and $F_B$ functions, from a certain number of parameters on, there are no non-resonant algebraic functions. 
However, for the Lauricella $F_C$ function, the situation is entirely different: for every number of parameters there are three infinite families of non-resonant algebraic functions.
The following Lemma is a generalization of Lemma~\ref{lem:F4_interlacing} in one direction.

\begin{lem}\label{lem:FC_2napexpoints}
Let $0 \leq a \leq b < 1$ and $0 \leq c_1, \ld, c_n < 1$.
Suppose that for all $I,J \subs \{1, \ld, n\}$ with $|I|$ even and $|J|$ odd it holds that
\begin{equation*}
b-1 \leq \sum_{i \in I} c_i - \frac{|I|}{2} < a \leq \sum_{j \in J} c_j - \frac{|J|-1}{2} < b.
\end{equation*}
Then there are $2^n$ apexpoints.
In particular, if $c_3=\ld=c_n=\half$, then there are $2^n$ apexpoints if 
\begin{equation*}
a \leq c_1, c_2 < b \leq c_1+\half, c_2+\half, c_1+c_2 < a+1 \leq c_1+c_2+\half < b+1.
\end{equation*}
\end{lem}

\begin{proof}
The second statement follows easily from the first.
For the first statement, we claim that all points $\x+\bal$ with $x_3, \ld, x_{n+2} \in \{-1,0\}$ and $x_k=|I|-\entier{\al_k+\sum_{i \in I} \al_i}$ are apexpoints, where $k=1,2$ and $I=\{i \in \{3, \ld, n+2\} \ | \ x_i=-1\}$.
Let $I'=\{i-2 \ | \ i \in I\}$.
Note that $x_k+\al_k+\sum_{i \in I} (x_i+\al_i) = \fp{\al_k+\sum_{i \in I} \al_i}$.
This equals $1-a+\sum_{i \in I'} c_i - \frac{|I|}{2}$ or $1-b+\sum_{i \in I'} c_i- \frac{|I|}{2}$ if $|I|$ is even, and $1-a+\sum_{i \in I'} c_i - \frac{|I|+1}{2}$ or $1-b+\sum_{i \in I'} c_i - \frac{|I|-1}{2}$ if $|I|$ is odd.

To show that $\x+\bal \in C(\A)$, we have to show that $x_k+\al_k+\sum_{j \in J} (x_j+\al_j) \geq 0$ for $k=1,2$ and $J \subs \{3, \ld, n+2\}$.
Since $x_j+\al_j \geq 0$ if and only if $j \not\in I$, it suffices to take $J=I$.
But then $x_k+\al_k+\sum_{j \in J} (x_j+\al_j) = \fp{\al_k+\sum_{i \in I} \al_i}$, which is clearly non-negative.
Since it is smaller than 1, we also have $\x+\bal-\e_1, \x+\bal-\e_2 \not\in C(\A)$.

Now we show that $\vect{y}=\x+\bal-\e_l \not\in C(\A)$ for $3 \leq l \leq n+2$, so we have to find $J$ such that $y_k+\al_k+\sum_{j \in J} (y_j+\al_j) < 0$.
Take $J=I \cup \{l\}$.
If $l \in I$, then $y_k+\al_k+\sum_{j \in J} (y_j+\al_j) = \fp{\al_k+\sum_{i \in I} \al_i}-1$, which is negative.
Hence we can assume that $l \not\in I$.
Then $y_k+\al_k+\sum_{j \in J} (y_j+\al_j) = \fp{\al_k+\sum_{i \in I} \al_i} + \al_l - 1$.
If $|I|$ is odd, take $k=1$.
Then $\fp{\al_k+\sum_{i \in I} \al_i} = 1-a+\sum_{i \in I'} c_i - \frac{|I|+1}{2}$, so
$y_k+\al_k+\sum_{j \in J} (y_j+\al_j) = 
-a+\sum_{i \in I'} c_i - \frac{|I|+1}{2} + c_{l-2} =
-a + \sum_{j \in J'} c_j - \frac{|J|}{2} < 0$.
Similarly, if $|I|$ is even, take $k=2$ to get
$\fp{\al_k+\sum_{i \in I} \al_i} = 1-b+\sum_{i \in I'} c_i- \frac{|I|}{2}$
and
$y_k+\al_k+\sum_{j \in J} (y_j+\al_j) = 
-b+\sum_{j \in J'} c_j- \frac{|J|-1}{2} < 0$. 

Finally, we show that $\vect{z}=\x+\bal-(\e_1+\e_2-\e_l) \not\in C(\A)$ for $3 \leq l \leq n+2$.
Take $J=I \setminus \{l\}$.
If $l \not\in I$, then $z_k+\al_k+\sum_{j \in J} (z_j+\al_j) = \fp{\al_k+\sum_{i \in I} \al_i}-1$, which is negative.
Let $l \in I$.
Then $z_k+\al_k+\sum_{j \in J} (z_j+\al_j) = \fp{\al_k+\sum_{i \in I} \al_i} -\al_1$. 
If $|I|$ is even, take $k=2$.
This gives $z_k+\al_k+\sum_{j \in J} (z_j+\al_j) = 
1-b+\sum_{i \in I'} c_i- \frac{|I|}{2} - c_{l-2} =
-b+\sum_{i \in I' \setminus \{l\}} c_i - \frac{|I' \setminus \{l\}|-1}{2} < 0$.
If $|I|$ is odd, we take $k=1$ to get
$\fp{\al_k+\sum_{i \in I} \al_i} -\al_1 =
1-a+\sum_{i \in I'} c_i - \frac{|I|+1}{2} - c_{1-2} =
-a + \sum_{i \in I' \setminus \{l\}} c_i - \frac{|I' \setminus \{l\}|}{2} < 0$.
\end{proof}

\begin{thm}\label{thm:FC_solutions}
For $n \geq 3$, $F_C(a,b,\vect{c} | \vect{z})$ is a non-resonant algebraic function if and only if up to permutations of $\{a, b\}$ and permutations of $\{c_1,\ld,c_n\}$, we have $c_3=\ld=c_n=\half$ and the tuple $(a,b,c_1,c_2)$ is conjugate to one of the tuples from Table~\ref{tab:FC_solutions}.
\end{thm}

\begin{proof}
First we show that all tuples from Table~\ref{tab:FC_solutions} indeed give non-resonant algebraic functions.
For non-resonance, it suffices to prove that $-a,-a+c_1,-a+c_2,-a+c_1+c_2,-b,-b+c_1,-b+c_2,-b+c_1+c_2$ are not half-integral.
This can easily be checked for all tuples.
To prove that the functions are algebraic, we use the second statement of Lemma~\ref{lem:FC_2napexpoints} to show that all conjugates have $2^n$ apexpoints. 
Again, this is an easy check.

Now we show that all non-resonant algebraic functions have parameters from Table~\ref{tab:FC_solutions}.
For $n=3$, we use the interlacing condition from Lemma~\ref{lem:F4_interlacing} and the fact that both $(a,b,c_1,c_2)$ and $(a,b,c_1,c_3)$ must be $F_4$ tuples. 
If both $(a,b,c_1,c_2)$ and $(a,b,c_1,c_3)$ are of type 1, then we have, up to permutations of $\{c_1,c_2,c_3\}$, $(a,b,c_1,c_2,c_3) \in \{(r,-r,\half,\half,\half), (r,r+\half,\half,\half,\half), (r,r+\half,\half,\half,2r), (r,r+\half,\half,2r,2r)\} \pmod \Z$.
The first three give non-resonant algebraic functions.
Let $(a,b,c_1,c_2,c_3) = (r,r+\half,\half,2r,2r)$.
$(a,b,c_2,c_3)$ must also be an $F_4$ tuple, but it is not of type 1 (unless $r=\pm \frac{1}{4}$, in which case it equals $(r,-r,\half,\half)$).
Hence it is of type 2, so the denominator of $2r$ is at most 5.
If the denominator of $r$ is 4, then the tuple is $(r,-r,\half,\half,\half)$ and if the denominator of $r$ equals 3 or 6, then the function will be resonant.
Therefore, we can assume that $r$ has denominator 5, 8 or 10.
Using the interlacing condition, we easily compute all algebraic functions.

If one of the tuples $(a,b,c_1,c_2)$ and $(a,b,c_1,c_3)$ is of type 1, then the parameter has denominator at most 60, so there are finitely many possibilities.
The same holds if both $(a,b,c_1,c_2)$ and $(a,b,c_1,c_3)$ are of type 2.
This gives the 720 conjugates of the tuples in Table~\ref{tab:FC_solutions}.

Finally, let $n \geq 4$, and suppose that all non-resonant algebraic functions in $n-1$ variables are given by Table~\ref{tab:FC_solutions}.
Let $(a,b,c_1,\ld,c_n)$ correspond to a non-resonant algebraic function.
Of each $n-1$ $c_i$'s, at least $n-3$ have to be equal to $\half$.
Hence at least $n-2$ of $c_1, \ld, c_n$ are equal to $\half$, so we can assume that $c_3=\ld=c_n=\half$.
Since $(a,b,c_1,\ld,c_{n-1})$ must also give a non-resonant algebraic function, $(a,b,c_1,c_2)$ must be one of the tuples in Table~\ref{tab:FC_solutions}.
\end{proof}

\renewcommand{\arraystretch}{2}
\begin{table}[htb]
\begin{center}
\caption{The tuples $(a,b,c_1,c_2)$ such that $F_C(a,b,c_1,c_2,\half, \ld, \half|\vect{z})$ is non-resonant and algebraic} 
\label{tab:FC_solutions} 
\begin{tabularx}{\textwidth}{XXXXX} 
\hline
$(r,		-r,		\frac{1}{2}, 	\frac{1}{2})$ &  
$(r, 		r+\frac{1}{2}, 	\frac{1}{2}, 	\frac{1}{2})$ &
$(r, 		r+\frac{1}{2},	\frac{1}{2}, 	2r)$ &
\multicolumn{2}{l}{with $r \in \parset$} \\

$(\frac{1}{4}, 	\frac{3}{4}, 	\frac{1}{2}, 	\frac{1}{3})$ &
$(\frac{1}{4}, 	\frac{3}{4}, 	\frac{1}{3}, 	\frac{2}{3})$ &
$(\frac{1}{4}, 	\frac{7}{12}, 	\frac{1}{2}, 	\frac{1}{2})$ &
$(\frac{1}{4}, 	\frac{7}{12}, 	\frac{1}{2}, 	\frac{1}{3})$ &
$(\frac{1}{12}, \frac{7}{12}, 	\frac{1}{2}, 	\frac{1}{3})$ \\

$(\frac{1}{12}, \frac{7}{12}, 	\frac{1}{3}, 	\frac{1}{3})$ &
$(\frac{1}{20}, \frac{11}{20}, 	\frac{1}{2}, 	\frac{1}{5})$ &
$(\frac{1}{20}, \frac{11}{20}, 	\frac{1}{2}, 	\frac{2}{5})$ &
$(\frac{1}{20}, \frac{11}{20}, 	\frac{1}{5}, 	\frac{2}{5})$ &
$(\frac{1}{20}, \frac{13}{20}, 	\frac{1}{2}, 	\frac{1}{2})$ \\

$(\frac{1}{20}, \frac{13}{20}, 	\frac{1}{2}, 	\frac{1}{5})$ &
$(\frac{1}{24}, \frac{13}{24}, 	\frac{1}{2}, 	\frac{1}{3})$ &
$(\frac{1}{24}, \frac{13}{24}, 	\frac{1}{2}, 	\frac{1}{4})$ &
$(\frac{1}{24}, \frac{13}{24}, 	\frac{1}{3}, 	\frac{1}{4})$ &
$(\frac{1}{24}, \frac{17}{24}, 	\frac{1}{2}, 	\frac{1}{2})$ \\

$(\frac{1}{24}, \frac{17}{24}, 	\frac{1}{2}, 	\frac{1}{4})$ &
$(\frac{1}{24}, \frac{19}{24}, 	\frac{1}{2}, 	\frac{1}{2})$ &
$(\frac{1}{24}, \frac{19}{24}, 	\frac{1}{2}, 	\frac{1}{3})$ &
$(\frac{1}{60}, \frac{31}{60}, 	\frac{1}{2}, 	\frac{1}{3})$ &
$(\frac{1}{60}, \frac{31}{60}, 	\frac{1}{2}, 	\frac{1}{5})$ \\

$(\frac{1}{60}, \frac{31}{60}, 	\frac{1}{3}, 	\frac{1}{5})$ &
$(\frac{1}{60}, \frac{41}{60}, 	\frac{1}{2}, 	\frac{1}{2})$ &
$(\frac{1}{60}, \frac{41}{60}, 	\frac{1}{2}, 	\frac{1}{5})$ &
$(\frac{1}{60},	\frac{49}{60}, 	\frac{1}{2}, 	\frac{1}{2})$ &
$(\frac{1}{60}, \frac{49}{60}, 	\frac{1}{2}, 	\frac{1}{3})$ \\
\hline
\end{tabularx}
\end{center}
\end{table}
\renewcommand{\arraystretch}{1}


\section{The Horn $G$ functions}

\subsection{The Horn $G_1$ function}\label{subsec:G1}

The $G_1$ function is defined by 
\begin{equation*}
G_1(a,b_1,b_2|x,y) = \sum_{m,n \geq 0} \frac{\poch{a}{m+n} \poch{b_1}{n-m} \poch{b_2}{m-n}}{m! n!} x^m y^n.
\end{equation*}
Hence the lattice is $\lat= \Z (-1, 1, -1, 1, 0) \oplus \Z (-1, -1, 1, 0, 1)$.
We choose $\A = \{\e_1, \e_2, \e_3, \e_1-\e_2+\e_3, \e_1+\e_2-\e_3\}$ and $\gam=(-a, -b_1, -b_2, 0, 0)$.
Then $\bal=(-a, -b_1, -b_2)$.

$\A$ lies in the hyperplane $x_1+x_2+x_3=1$.
By projecting $\A$ onto the $(x_1,x_2)$-plane, we get the set shown in Figure~\ref{fig:G1_A}.
The thick dots represent $\A$, the dark gray region is the set $Q(\A)$ and light gray region is a part of the set $C(\A)$.
From this Figure, it is clear that $Q(\A)$ has volume 3 and has a unimodular triangulation.
Hence $\A$ is saturated.
It is clear that 
\begin{equation*}
C(\A) = \{\x \in \R^3 \ | \ x_1 \geq 0, x_1+x_2 \geq 0, x_1+x_3 \geq 0, x_2+x_3 \geq 0\}.
\end{equation*}
Hence $G_1(a,b_1,b_2|x,y)$ is non-resonant if and only if $a$, $a+b_1$, $a+b_2$ and $b_1+b_2$ are non-integral.

\begin{lem}\label{lem:G1_interlacing}
There are 3 apexpoints if and only if either $\fp{a}+\fp{b_1} \leq 1$, $\fp{a}+\fp{b_2} \leq 1$ and $\fp{b_1}+\fp{b_2}>1$, or $\fp{a}+\fp{b_1}>1$, $\fp{a}+\fp{b_2}>1$ and $\fp{b_1}+\fp{b_2} \leq 1$.
\end{lem}

\begin{proof}
It is easy to show that there are 3 apexpoints if and only if $(\entier{\al_1+\al_2}, \entier{\al_1+\al_3}, \entier{\al_2+\al_3}) \in \{(0,0,1),(1,1,0)\}$.
Since $a \not\in \Z$, we have $\al_1=1-\fp{a}$.
We have either $\al_1+\al_2 \geq 1$ or $\al_2+\al_3 \geq 1$, so $\al_2>0$. 
Similarly, $\al_3>0$ and hence $\bal=(1-\fp{a},1-\fp{b_1},1-\fp{b_2})$.
\end{proof}

\begin{lem}\label{lem:G1_implies_F}
If $G_1(a,b_1,b_2|x,y)$ is non-resonant and algebraic, then $F(a,b_1,a+b_1+b_2|z)$ is irreducible and algebraic.
\end{lem}

\begin{proof}
Irreducibility follows from the interlacing condition for $G_1(a,b_1,b_2|x,y)$.

We show that the interlacing condition for $G_1(a,b_1,b_2|x,y)$ implies the interlacing condition for $F(a,b_1,a+b_1+b_2|z)$.
If $\fp{a}+\fp{b_1} \leq 1$, $\fp{a}+\fp{b_2} \leq 1$ and $\fp{b_1}+\fp{b_2}>1$, then $\fp{a+b_1+b_2} = \fp{a}+\fp{b_1}+\fp{b_2}-1$ and $\fp{a} \leq \fp{a}+\fp{b_1}+\fp{b_2}-1 < \fp{b_2}$.
The other case is similar.
\end{proof}

\begin{thm}\label{thm:G1_solutions}
$G_1(a,b_1,b_2|x,y)$ is non-resonant and algebraic if and only if $(a,b_1,b_2) \pmod \Z$ is one of the following:
$\pm (\frac{1}{6}, \half, \frac{2}{3})$,
$\pm (\frac{1}{6}, \frac{2}{3}, \half)$ and
$\pm (\frac{1}{6}, \frac{2}{3}, \frac{2}{3})$.
\end{thm}

\begin{proof}
We only have to consider $(a,b_1,b_2)$ such that $(a,b_1,a+b_1+b_2)$ is a Gauss triple.
Suppose that $(a,b_1,a+b_1+b_2)$ is a Gauss triple of type 1.
Then $(a,b_1,b_2) \in \{(r,-r,\half), (r,r+\half,-2r), (r,r+\half,\half) \pmod \Z$.
If $(a,b_1,b_2) = (r,-r,\half)$, then $a+b_1 \in \Z$, so $G_1(a,b_1,b_2|x,y)$ is resonant.
Suppose that $(a,b_1,b_2)=(r,r+\half,-2r)$.
Then the function is non-resonant if $r \neq \frac{1}{4}$.
The interlacing condition is satisfied if and only if $r \leq \frac{1}{4}$ or $r>\frac{3}{4}$.
Hence we must have $\fp{kr} \leq \frac{1}{4}$ or $\fp{kr}>\frac{3}{4}$ for all $k$ coprime with the denominator of $r$.
By Lemma~\ref{lem:closetohalf}, this is possible only if the denominator of $r$ is 6, so $r=\frac{1}{6}$ or $r=\frac{5}{6}$.
This gives the solutions $(a,b_1,b_2) = \pm(\frac{1}{6}, \frac{2}{3}, \frac{2}{3})$.
If $(a,b_1,b_2)=(r,r+\half,\half)$, then the interlacing condition again reduces to $\fp{kr} \leq \frac{1}{4}$ or $\fp{kr}>\frac{3}{4}$, so the solutions are $\pm(\frac{1}{6}, \frac{2}{3}, \half)$.
If $(a,b_1,b_2)$ is a Gauss triple of type 2, then there are only finitely many possibilities.
This gives two more non-resonant algebraic functions, with $(a,b_1,b_2)=\pm(\frac{1}{6}, \half, \frac{2}{3})$.
\end{proof}


\subsection{The Horn $G_2$ function}\label{subsec:G2}

The $G_2$ function is defined by 
\begin{equation*}
G_2(a_1,a_2,b_1,b_2|x,y) = \sum_{m,n \geq 0} \frac{\poch{a_1}{m} \poch{a_2}{n} \poch{b_1}{n-m} \poch{b_2}{m-n}}{m! n!} x^m y^n. \\
\end{equation*}
Hence the lattice is $\lat= \Z (-1, 0, 1, -1, 1, 0) \oplus \Z (0, -1, -1, 1, 0, 1)$.
We choose $\A = \{\e_1,\e_2,\e_3,\e_4,\e_1-\e_3+\e_4,\e_2+\e_3-\e_4\}$.
With $\gam=(-a_1, -a_2, -b_1, -b_2, 0, 0)$, we get $\bal=(-a_1, -a_2, -b_1, -b_2)$.
The function $f:\Z^4 \rightarrow \Z^4: x \mapsto (x_2,x_3,x_3+x_4,x_1-x_3)$ maps $F_1$ to $G_2$ as in Remark~\ref{rmk:transformations}.
Since $f^{-1}:\Z^4 \rightarrow \Z^4: x \mapsto (x_2+x_4,x_1,x_2,-x_2+x_3)$, we get:

\begin{lem}\label{lem:G2_irreducible}
$G_2(a_1,a_2,b_1,b_2|x,y)$ is non-resonant if and only if $a_1$, $a_2$, $a_1+b_1$, $a_2+b_2$ and $b_1+b_2$ are non-integral.
\end{lem}

\begin{thm}\label{thm:G2_solutions}
$G_2(a_1,a_2,b_1,b_2|x,y)$ is non-resonant and algebraic if and only if $(a_1,a_2,b_1,b_2)$ or $(a_2,a_1,b_2,b_1) \pmod \Z$ equals 
$\pm(\frac{1}{3}, \frac{1}{6}, \half, \frac{2}{3})$,
$\pm(\frac{1}{6}, \frac{1}{6}, \half, \frac{2}{3})$ or
$\pm(\frac{1}{6}, \frac{1}{6}, \frac{2}{3}, \frac{2}{3})$.
\end{thm}


\begin{figure}
\centering
\psset{unit=0.75cm}
\subfigure[The sets $\A$, $Q(\A)$ and $C(\A)$ for $G_1$]{
\label{fig:G1_A}
\begin{picture}(3,4.5)(-1,-2)
\pspolygon[linestyle=none,fillstyle=solid,fillcolor=lightgray](0,0)(2,-2)(2,2)(0,2)
\pspolygon[linewidth=2pt,fillstyle=solid,fillcolor=darkgray](0,0)(1,-1)(1,1)(0,1)
\psline(0,2)(0,0)(2,-2)
\psline[linewidth=2pt](0,0)(1,1)
\psline[linewidth=2pt](0,0)(1,0)
\psdots[dotsize=5pt](1,0)(0,1)(0,0)(1,-1)(1,1)
\psgrid[subgriddiv=1,griddots=1,gridwidth=2pt,gridlabels=0](0,0)(-1,-2)(2,2)
\psaxes[ticks=none,linewidth=0.5pt,labels=none](0,0)(-1,-2)(2,2)
\end{picture}}
\subfigure[The sets $\A$, $Q(\A)$ and $C(\A)$ for $G_3$]{
\label{fig:G3_A}
\begin{picture}(6,3)(-1,-0.5)
\pspolygon[linestyle=none,fillstyle=solid,fillcolor=lightgray](0,2)(2,0)(6,2)
\psline(0,2)(2,0)(6,2)
\psline[linewidth=2pt](1,1)(4,1)
\psdots[dotsize=5pt](1,1)(2,1)(3,1)(4,1)
\psgrid[subgriddiv=1,griddots=1,gridwidth=2pt,gridlabels=0](2,0)(0,-0.5)(6,2)
\psaxes[ticks=none,linewidth=0.5pt,labels=none](2,0)(0,-0.5)(6,2)
\end{picture}}
\subfigure[The interlacing condition for $G_3$]{
\label{fig:G3_interlacing_al}
\begin{picture}(5,4)(-1.5,-0.5)
\pspolygon[fillstyle=solid,fillcolor=lightgray,linestyle=none](0,0)(2,1)(3,0)
\pspolygon[fillstyle=solid,fillcolor=lightgray,linestyle=none](0,3)(1,2)(3,3)
\psaxes[Dx=0.5,Dy=0.5,dx=1.5,dy=1.5](3,3)
\psline[linestyle=dotted,linewidth=1.5pt](3,0)(3,3)
\psline[linestyle=dotted,linewidth=1.5pt](0,3)(3,3)
\psline(0,3)(1,2)(3,3)
\psline[linestyle=dashed](0,0)(2,1)(3,0)
\psline[linestyle=dotted,linewidth=1pt](0,1.5)(1,2)(2,1)(3,1.5)
\end{picture}}
\caption{Pictures for the $G$ functions}
\psset{unit=1cm}
\end{figure}
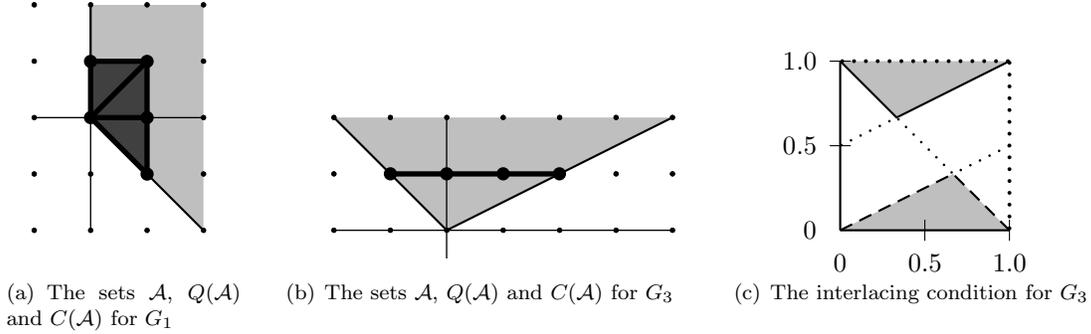


\subsection{The Horn $G_3$ function}\label{subsec:G3}

The $G_3$ function is defined by 
\begin{equation*}
G_3(a_1,a_2|x,y) = \sum_{m,n \geq 0} \frac{\poch{a_1}{2n-m} \poch{a_2}{2m-n}}{m! n!} x^m y^n. \\ 
\end{equation*}
Hence the lattice is $\lat= \Z (1, -2, 1, 0) \oplus \Z (-2, 1, 0, 1)$.
We choose $\A = \{\e_1+\e_2, \e_2, -\e_1+\e_2, 2\e_1+\e_2\}$ and $\gam=(-a_1,-a_2,0,0)$.
This gives $\bal=(-a_1, -a_1-a_2)$.

In Figure~\ref{fig:G3_A}, the thick dots represent the set $\A$; the thick line is $Q(\A)$ and the gray region is a part of the cone $C(\A)$. 
It is clear that $Q(\A)$ has volume 3 and has a unimodular triangulation, so $\A$ is saturated.
We have 
\begin{equation*}
C(\A) = \{\x \in \R^2 \ | \ x_1+x_2 \geq 0, -x_1+2x_2 \geq 0\}.
\end{equation*}
It follows that $G_3(a_1,a_2|x,y)$ is non-resonant if and only if $2a_1+a_2$ and $a_1+2a_2$ are non-integral.

One easily computes that there are 3 apexpoints if and only if either $-1 \leq -\al_1+2 \al_2 < 0 \leq \al_1+\al_2 < 1$ or $1 \leq -\al_1+2\al_2, \al_1+\al_2 < 2$. 
Figure~\ref{fig:G3_interlacing_al} gives a graphical interpretation of the interlacing condition.
There are 3 apexpoints if and only if $\bal$ lies in the gray region.
This Figure also gives a idea how to find all algebraic functions:
if some multiple of $\al_2$ is close enough to $\half$, then the function will not be algebraic.
Hence the denominator of $\al_2$ must be smal..
Furthermore, if the numerator of a multiple of $\al_1$ equals 1, then $\al_2$ must be sufficiently small.
Hence the denominator of $\al_1$ can not be too big.

\begin{thm}\label{G3_solutions}
$G_3(a_1,a_2|x,y)$ is non-resonant and algebraic if and only if $a_1+a_2 \in \Z$ or, up to equivalence modulo $\Z$, 
$(a_1,a_2) \in \{(\half,\frac{1}{3}), (\half,\frac{2}{3}), (\frac{1}{3}, \half), (\frac{2}{3}, \half)\}$.
\end{thm}

\begin{proof}
Write $\al_1=\frac{p}{q}$ and $\al_2=\frac{u}{v}$ with $\gcd(p,q)=\gcd(u,v)=1$, $0 \leq p < q$ and $0 \leq u < v$.
It follows immediately from the interlacing condition that $\al_1 \neq 0$ and $\al_2 \neq \half$.
Furthermore, if $\al_2=0$ (i.e.\ $a_1+a_2 \in \Z$), then the interlacing condition is satisfied for all $\al_1$.
Therefore, we will assume that $q \geq 2$, $v \geq 3$ and $p,u \neq 0$.

The interlacing condition is not satisfied if there exists $k$ such that $\fp{k \al_2} \in [\frac{1}{3}, \half)$ and $\gcd(k,qv)=1$.
By Lemma~\ref{lem:closetohalf}, such $k$ exists unless $v \in \{4,6,10\}$.
Hence we can assume that $v \in \{4,6,10\}$.
Choose $l$ such that $lp \equiv 1 \pmod q$ and $\gcd(l,qv)=1$.
Then $\fp{l \al_1} = \frac{1}{q}$.
Write $\fp{l \al_2} = \frac{t}{v}$ with $0 \leq t < v$.
If $\frac{t}{v} < \frac{1}{3}$, then it must hold that $2\fp{l \al_2}<\fp{l \al_1}$, i.e. $2tq<v$.
This gives $(\fp{l \al_1}, \fp{l \al_2}) \in 
\{(\frac{1}{2},\frac{1}{6}), (\frac{1}{2},\frac{1}{10}), (\frac{1}{3},\frac{1}{10}), (\frac{1}{4},\frac{1}{10})\}$.
If $\frac{t}{v} \geq \frac{2}{3}$, then we must have $\fp{l \al_1}+\fp{l \al_2} \geq 1$, so $q \leq 4$ if $v=2$, $q \leq 6$ if $v=6$, $q \leq 3$ if $\frac{t}{v}=\frac{7}{10}$ and $q \leq 10$ if $\frac{t}{v}=\frac{9}{10}$.
Now one easily checks that all conjugates satisfy the interlacing condition if and only if $(\fp{l \al_1}, \fp{l \al_2}) \in \{(\half, \frac{1}{6}), (\half, \frac{5}{6}), (\frac{1}{3}, \frac{5}{6})\}$.

It is easy to check that these parameters, as well as $\al_2=0$, give non-resonant functions.
Hence the non-resonant algebraic functions are given by the orbits of these parameters.
\end{proof}


\section{The Horn $H$ functions}

\subsection{The Horn $H_1$ function}\label{subsec:H1}

The $H_1$ function is defined by 
\begin{equation*}
H_1(a,b,c,d|x,y) =
\sum_{m,n \geq 0} \frac{\poch{a}{m-n} \poch{b}{m+n} \poch{c}{n}}{\poch{d}{m} m! n!} x^m y^n. 
\end{equation*}
Hence the lattice is $\lat= \Z (-1, -1, 0, 1, 1, 0) \oplus \Z (1, -1, -1, 0, 0, 1)$.
Take $\A=\{\e_1, \e_2, \e_3, \e_4, \e_1+\e_2-\e_4, -\e_1+\e_2+\e_3\}$ and $\bal=(-a, -b, -c, d-1)$.
Using Lemma~\ref{lem:convex_hull}, it is easily checked that $V_1=\{\e_1, \e_2, \e_3, \e_4\}$, $V_2=\{\e_1, \e_2, \e_3, \e_1+\e_2-\e_4\}$, $V_3=\{-\e_1+\e_2+\e_3, \e_2, \e_3, \e_4\}$ and $V_4 = \{-\e_1+\e_2+\e_3, \e_2, \e_3, \e_1+\e_2-\e_4\}$ give a unimodular triangulation of $Q(\A)$. 
This implies the following Lemma:

\begin{lem}\label{lem:H1_prop_sd}
$\A$ is saturated, the volume of $Q(\A)$ is 4 and 
\begin{equation*}
\begin{split}
C(\A) = \{\x \in \R^4\ | \ & 
x_2 \geq 0, x_3 \geq 0, x_1+x_2 \geq 0, x_1+x_3 \geq 0, \\
 & x_2+x_4 \geq 0, x_1+x_2+2x_4 \geq 0, x_1+x_3+x_4 \geq 0\}. 
\end{split}
\end{equation*}
$H_1(a,b,c,d|x,y)$ is non-resonant if and only if $b$, $c$, $a+b$, $a+c$, $d-b$, $d-a-c$ and $2d-a-b$ are non-integral.
\end{lem}

\begin{lem}\label{lem:H1_interlacing}
Suppose that $H_1(a,b,c,d|x,y)$ is non-resonant.
Then there are 4 apexpoints if and only if one of the following conditions holds:
\begin{equation*}
\fp{a}+\fp{c} \leq \fp{d} \andd \fp{a}+\fp{b}>1, 2\fp{d} 
\end{equation*}
or
\begin{equation*}
\fp{d}+1 < \fp{a} + \fp{c} \andd \fp{a}+\fp{b} \leq 1, 2\fp{d} 
\end{equation*}
or
\begin{equation*}
\fp{a}+\fp{c} -1, \fp{b} \leq \fp{d} \andd 2\fp{d} < \fp{a}+\fp{b} \andd \min(\fp{b},\fp{c}) \leq 1 -\fp{a} < \max(\fp{b},\fp{c}) 
\end{equation*}
or
\begin{equation*}
\fp{d} < \fp{a}+\fp{c}, \fp{b} \andd \fp{a}+\fp{b} \leq 2\fp{d} \andd \min(\fp{b},\fp{c}) \leq 1 -\fp{a} < \max(\fp{b},\fp{c}) 
\end{equation*}
\end{lem}

\begin{proof}
One easily computes that there are 4 apexpoints if and only if $(\entier{\al_1+\al_2}$, $\entier{\al_1+\al_3}$, $\entier{\al_2+\al_4}$, $\entier{\al_1+\al_2+2\al_4}$, $\entier{\al_1+\al_3+\al_4}) \in$ $\{(1,0,0,2,1)$, $(0,1,0,2,1)$, $(0,1,0,1,2)$, $(1,0,1,2,1)$, $(1,0,1,1,1)$, $(0,1,1,1,1)\}$. 
Since either $\al_1+\al_2 \geq 1$ or $\al_1+\al_3 \geq 0$, we have $a \not\in \Z$ and $\bal=(1-\fp{a},1-\fp{b},1-\fp{c},\fp{d})$.
\end{proof}

\begin{lem}\label{lem:H1_implies_f}
Suppose that $H_1(a,b,c,d|x,y)$ is non-resonant and algebraic.
Then $F(a,b,d|z)$ and $F(b-d,c,d-a|z)$ are irreducible and algebraic.
\end{lem}

\begin{proof}
Irreducibility follows from the interlacing condition and non-resonance for $H_1(a,b,c,d|x,y)$.
Note that $F(a,b,d|z) = H_1(a,b,c,d|z,0)$.
Hence it suffices to show that the interlacing condition for $H_1(a,b,c,d|x,y)$ implies the interlacing condition for $F(b-d,c,d-a|z)$.

If $\fp{a}+\fp{c} \leq \fp{d}$ and $\fp{a}+\fp{b}>1, 2\fp{d}$, then $\fp{d-a}=\fp{d}-\fp{a}$ and $\fp{b} > 2\fp{d}-\fp{a} \geq \fp{d}$, so $\fp{b-d}=\fp{b}-\fp{d}$.
Hence $\fp{c} \leq \fp{d}-\fp{a} = \fp{d-a} < \fp{b}-\fp{d} = \fp{b-d}$.

If $\fp{a}+\fp{c} -1, \fp{b} \leq \fp{d}$, $2\fp{d} < \fp{a}+\fp{b}$ and $\min(\fp{b},\fp{c}) \leq 1 -\fp{a} < \max(\fp{b},\fp{c})$, then $\fp{b-d}=\fp{b}-\fp{d}+1$ and $\fp{a} > 2\fp{d}-\fp{b} \geq \fp{d}$, so $\fp{d-a}=\fp{d}-\fp{a}+1$.
Hence $\fp{c} \leq \fp{d}-\fp{a}+1 = \fp{d-a} < \fp{b}-\fp{d}+1 = \fp{b-d}$.

The other two cases are similar.
\end{proof}

\begin{thm}\label{thm:H1_solutions}
$H_1(a,b,c,d|x,y)$ is non-resonant and algebraic if and only if $(a,b,c,d)$ is, up to equivalence modulo $\Z$,  one of the following:
$\pm(\frac{1}{3}, \frac{5}{6}, \half, \frac{2}{3})$, 
$\pm(\frac{1}{4}, \frac{7}{12}, \frac{5}{6}, \half)$ and
$\pm(\frac{1}{4}, \frac{11}{12}, \frac{1}{6}, \half)$. 
\end{thm}

\begin{proof}
If $H_1(a,b,c,d|x,y)$ is non-resonant and algebraic, then $(a,b,d)$ and $(b-d,c,d-a)$ are Gauss triples.
Suppose that $(a,b,d)$ is of type 1.
Then there exists $r \in \parset$ such that $(b-d,d-a) \in \{(-r+\half, -r+\half), (r, -r+\half), (-r+\half,r)\} \pmod \Z$.
Note that $d-a \neq \half$. 
Hence if $(b-d,c,d-a)$ is also of type 1, then there exists $s \in \parset$ such that $(b-d,c,d-a) = (s,s+\half,2s)$.
Modulo $\Z$, we have $d-a \equiv 2(b-d)$.
This implies that $(b-d,d-a) = (r, -r+\half)$ with $r=\pm \frac{1}{6}$ or $(b-d,d-a) = (-r+\half,r)$ with $r=\pm \frac{1}{3}$.
Hence the possibilities for $(a,b,c,d)$ are $\pm(\frac{1}{6}, \frac{2}{3}, \frac{2}{3}, \half)$ and $\pm(\frac{1}{3}, \frac{5}{6}, \frac{2}{3}, \frac{2}{3})$.
However, they don't satisfy the interlacing condition.
If $(b-d,c,d-a)$ is of type 2, then the denominator of $d-a$ is at most 5.
Hence the denominator of $r$ is at most 10.
We check all possibilities, and find as solutions the tuples $\pm(\frac{1}{3}, \frac{5}{6}, \half, \frac{2}{3})$.

Now suppose that $(a,b,d)$ is of type 2.
Then the denominator of $b-d$ is at most 60.
If $(b-d,c,d-a)$ is of type 1, then it is of the form $(s,-s,\half), (s,s+\half,\half)$ or $(s,s+\half,2s)$ where the denominator of $s$ is at most 60.
This gives finitely many possibilities, and there are no algebraic functions.

Finally, if both $(a,b,d)$ and $(b-d,c,d-a)$ are of type 2, then the solutions are $\pm(\frac{1}{4}, \frac{7}{12}, \frac{5}{6}, \half)$ and $\pm(\frac{1}{4}, \frac{11}{12}, \frac{1}{6}, \half)$.
\end{proof}


\subsection{The Horn $H_2$ function}\label{subsec:H2}

The $H_2$ function is defined by
\begin{equation*}
H_2(a,b,c,d,e|x,y) = \sum_{m,n \geq 0} \frac{\poch{a}{m-n} \poch{b}{m} \poch{c}{n} \poch{d}{n}}{\poch{e}{m} m! n!} x^m y^n. 
\end{equation*}
The lattice is $\lat = \Z (-1, -1, 0, 0, 1, 1, 0) \oplus \Z (1, 0, -1, -1, 0, 0, 1)$.
Take $\A = \{\e_1, \e_2, \e_3, \e_4, \e_5, \e_1+\e_2-\e_5, -\e_1+\e_3+\e_4\}$ and $\gam = (-a,-b,-c,-d,e-1,0,0)$.
Then $\bal=(-a,-b,-c,-d,e-1)$.
The function $f: x \mapsto (x_1+x_3, x_2, x_3+x_5, x_3, x_4)$ maps $F_2$ to $H_2$.

\begin{lem}\label{lem:H2_irreducible}
$H_2(a,b,c,d,e|x,y)$ is non-resonant if and only if $b, c, d$, $a+c, a+d$, $e-b$, $e-a-c$ and $e-a-d$ are non-integral.
\end{lem}

\begin{thm}\label{thm:H2_solutions}
$H_2(a,b,c,d,e|x,y)$ is non-resonant and algebraic if and only if, up to equivalence modulo $\Z$ and permutations of $\{c,d\}$, $(a,b,c,d,e)$ is conjugate to one of the following tuples:
$(\half, \frac{1}{6}, \frac{5}{12}, \frac{11}{12}, \frac{1}{3})$,
$(\frac{1}{3}, \frac{5}{6}, \frac{1}{4}, \frac{3}{4}, \frac{2}{3})$,
$(\frac{1}{3}, \frac{5}{6}, \frac{1}{6}, \frac{5}{6}, \frac{2}{3})$,
$(\frac{1}{3}, \frac{5}{6}, \frac{1}{10}, \frac{9}{10}, \frac{2}{3})$, 
$(\frac{1}{4}, \frac{3}{4}, \frac{1}{6}, \frac{5}{6}, \half)$,
$(\frac{1}{4}, \frac{7}{12}, \frac{1}{6}, \frac{5}{6}, \half)$,
$(\frac{1}{5}, \frac{7}{10}, \frac{1}{6}, \frac{5}{6}, \frac{2}{5})$,
$(\frac{1}{5}, \frac{7}{10}, \frac{1}{10}, \frac{9}{10}, \frac{2}{5})$ and
$(\frac{1}{6}, \frac{5}{6}, \frac{5}{12}, \frac{11}{12}, \frac{2}{3})$.
\end{thm}


\subsection{The Horn $H_3$ function}\label{subsec:H3}

The $H_3$ function is defined by 
\begin{equation*}
H_3(a,b,c|x,y) = \sum_{m,n \geq 0} \frac{\poch{a}{2m+n} \poch{b}{n}}{\poch{c}{m+n} m! n!} x^m y^n. 
\end{equation*}
Hence the lattice is equals 
$\lat= \Z (-2, 0, 1, 1, 0) \oplus \Z (-1, -1, 1, 0, 1)$.
Choose $\A = \{\e_1,\e_2,\e_3,2\e_1-\e_3,\e_1+\e_2-\e_3\}$ and $\gam=(-a,-b,c-1,0,0)$.
Then $\bal=(-a,-b,c-1)$.
Consider $f: x \mapsto (x_1+x_3, x_2+x_3, -x_3)$.
This function maps the set $\A$ of $G_1$ to the set $\A$ of $H_3$.
Hence:

\begin{lem}\label{lem:H3_irreducible}
$H_3(a,b,c|x,y)$ is non-resonant if and only if $a$, $b$, $c-a$ and $2c-a-b$ are non-integral.
\end{lem}

\begin{thm}\label{thm:H3_solutions_abc}
$H_3(a,b,c|x,y)$ is non-resonant and algebraic if and only if $(a,b,c)$ is one of the following:
$\pm(\frac{1}{3}, \frac{5}{6}, \half)$,
$\pm(\frac{1}{6}, \frac{2}{3}, \frac{1}{3})$ and
$\pm(\frac{1}{6}, \frac{5}{6}, \frac{1}{3})$. 
\end{thm}


\subsection{The Horn $H_4$ function}\label{subsec:H4}

The $H_4$ function is defined by
\begin{equation*}
H_4(a,b,c,d|x,y) = \sum_{m,n \geq 0} \frac{\poch{a}{2m+n} \poch{b}{n}}{\poch{c}{m} \poch{d}{n} m! n!} x^m y^n. 
\end{equation*}
Hence the lattice is $\lat = \Z (-2,0,1,0,1,0) \oplus \Z (-1,-1,0,1,0,1)$.
Take $\A = \{\e_1, \e_2, \e_3, \e_4, 2\e_1-\e_3, \e_1+\e_2-\e_4\}$.
With $\gam = (-a,-b,c-1,d-1,0,0)$ we get $\bal=(-a,-b,c-1,d-1)$. 
The sets $V_1=\{\e_1, \e_2, \e_3, \e_4\}$, $V_2=\{\e_1, \e_2, \e_3, \e_1+\e_2-\e_4\}$, $V_3=\{\e_1, \e_2, \e_4, 2\e_1-\e_3\}$ and $V_4=\{\e_1, \e_2, 2\e_1-\e_3, \e_1+\e_2-\e_4\}$ define a unimodular triangulation of $Q(\A)$.
Hence:

\begin{lem}\label{lem:H4_prop_trian}
$\A$ is saturated, the volume of $Q(\A)$ is 4 and 
\begin{equation*}
C(\A) = \{\x \in \R^4 \ | \ x_1, x_2 \geq 0, x_1+2x_3 \geq 0, x_1+x_4 \geq 0, x_2+x_4 \geq 0, x_1+2x_3+x_4 \geq 0\}.
\end{equation*}
$H_4(a,b,c,d|x,y)$ is non-resonant if and only if $a$, $b$, $2c-a$, $d-a$, $d-b$ and $2c+d-a$ are non-integral.
\end{lem}

\begin{lem}\label{lem:H4_interlacing_alpha}
Suppose that $H_4(a,b,c,d|x,y)$ is non-resonant.
Then there are 4 apexpoints if and only if either $\fp{a} \leq \fp{d} < \fp{b}$ and $2\fp{c} < \fp{a}+1 \leq 2\fp{c}+\fp{d}$, or $\fp{b} \leq \fp{d} < \fp{a} \leq 2\fp{c}$ and $2\fp{c}+\fp{d} < \fp{a}+1$.
\end{lem}

\begin{proof}
There are 4 apexpoints if and only if $(\entier{\al_1+2\al_3}, \entier{\al_1+\al_4}, \entier{\al_2+\al_4}, \entier{\al_1+2\al_3+\al_4})$ equals $(1,1,0,2)$ or $(1,0,1,1)$.
\end{proof}

\begin{lem}\label{lem:H4_f(a-2c)}
If $H_4(a,b,c,d|x,y)$ is non-resonant and algebraic, then $F(\frac{a}{2},\frac{a+1}{2},c|z)$, $F(a,b,d|z)$ and $F(b,a-2c,d|z)$ are irreducible and algebraic.
\end{lem}

\begin{proof}
Irreducibility is clear.
Since we have $F(\frac{a}{2},\frac{a+1}{2},c|z) = H_4(a,b,c,d|\frac{z}{4},0)$ and $F(a,b,d|z) = H_4(a,b,c,d|0,z)$, these functions are algebraic.

If $\fp{a} \leq \fp{d} < \fp{b}$ and $2\fp{c} < \fp{a}+1 \leq 2\fp{c}+\fp{d}$,
then we have $\fp{a}-2\fp{c}+1>0$, so $\fp{a-2c} \leq \fp{a}-2\fp{c}+1<\fp{d}<\fp{b}$.
If $\fp{b} \leq \fp{d} < \fp{a} \leq 2\fp{c}$ and $2\fp{c}+\fp{d} < \fp{a}+1$,
then $\fp{a}-2\fp{c}<0$ (since $a-2c \not \in \Z$), and hence $\fp{b} \leq \fp{d} < \fp{a}-2\fp{c}+1 \leq \fp{a-2c}$.
Hence the interlacing condition for $H_4(a,b,c,d|x,y)$ implies the interlacing condition for $F(b,a-2c,d|z)$.
\end{proof}

\begin{thm}\label{thm:H4_solutions}
$H_4(a,b,c,d|x,y)$ is non-resonant and algebraic if and only if $(a,b,c,d)$ is conjugate to one of the tuples in Table~\ref{tab:H4_solutions}.
\end{thm}

\begin{proof}
We only have to consider $(a,b,c,d)$ such that $(\frac{a}{2},\frac{a+1}{2},c)$, $(a,b,d)$ and $(b,a-2c,d)$ are Gauss triples.
Suppose that $(\frac{a}{2},\frac{a+1}{2},c)$ is of type 1. Then either $c=\half$ or $c=a$.
If $(a,b,d)$ is also of type 1, then there exists $r \in \parset$ such that $(a,b,c,d) \pmod \Z \in$ $\{(r,-r,\half,\half)$, $(r,-r,r,\half)$, $(r,r+\half,\half,\half)$, $(r,r+\half,r,\half)$, $(r,r+\half,\half,2r)$, $(r,r+\half,r,2r)\}$. This gives $(b,a-2c,d) \pmod \Z \in$ $\{(-r,\pm r,\half)$, $(r+\half,\pm r,\half)$, $(r+\half,\pm r,\half)$.
One easily checkes that precisely the triples with $a-2c$ are indeed Gauss triples.
In those cases, all conjugates of $(a,b,c,d)$ satisfy the interlacing condition.
If $(a,b,d)$ is of type 2, we can just check the interlacing condition for all tuples $(a,b,\half,d)$ and $(a,b,a,d)$.
This gives 408 solutions.

If $(\frac{a}{2},\frac{a+1}{2},c)$ is of type 2 and $(a,b,d)$ is of type 1, then the denominator of $a$ is at most 30.
We check the interlacing condition for all possibilities and find 8 solutions.

Finally, if both $(\frac{a}{2},\frac{a+1}{2},c)$ and $(a,b,d)$ are of type 2, then there are finitely many possibilities.
This gives another 36 solutions.
Of all 452 solutions, the smallest conjugate is given in Table~\ref{tab:H4_solutions}.
\end{proof}

\renewcommand{\arraystretch}{2}
\begin{table}[htb]
\begin{center}
\caption{The tuples $(a,b,c,d)$ such that $H_4(a,b,c,d|x,y)$ is non-resonant and algebraic} 
\label{tab:H4_solutions} 
\begin{tabularx}{\textwidth}{XXXXXX}
\hline
$(r,-r,\half,\half)$ & 
$(r,r+\half,\half,\half)$ &
\multicolumn{2}{l}{$(r,r+\half,\half,2r)$} &
\multicolumn{2}{l}{with $r \in \parset$} \\

$(\half, \frac{1}{6}, \half, \frac{1}{3})$ &
$(\half, \frac{1}{6}, \frac{1}{3}, \frac{1}{3})$ &
$(\frac{1}{4}, \frac{3}{4}, \half, \frac{1}{3})$ &
$(\frac{1}{4}, \frac{7}{12}, \half, \half)$ &
$(\frac{1}{4}, \frac{7}{12}, \half, \frac{1}{3})$ &
$(\frac{1}{6}, \half, \half, \frac{1}{3})$ \\

$(\frac{1}{6}, \frac{5}{6}, \half, \frac{1}{3})$ & 
$(\frac{1}{6}, \frac{5}{6}, \half, \frac{2}{3})$ &
$(\frac{1}{6}, \frac{5}{6}, \half, \frac{1}{4})$ & 
$(\frac{1}{6}, \frac{5}{6}, \half, \frac{1}{5})$ &
$(\frac{1}{6}, \frac{5}{6}, \frac{1}{3}, \frac{2}{3})$ &
$(\frac{1}{6}, \frac{5}{12}, \half, \frac{1}{3})$ \\

$(\frac{1}{6}, \frac{5}{12}, \half, \frac{1}{4})$ &
$(\frac{1}{6}, \frac{11}{30}, \half, \frac{1}{3})$ & 
$(\frac{1}{6}, \frac{11}{30}, \half, \frac{1}{5})$ &
$(\frac{1}{10}, \frac{3}{10}, \half, \frac{1}{5})$ &
$(\frac{1}{10}, \frac{7}{10}, \half, \frac{2}{5})$ &
$(\frac{1}{10}, \frac{7}{10}, \frac{2}{5}, \frac{2}{5})$ \\

$(\frac{1}{10}, \frac{9}{10}, \half, \frac{1}{3})$ &
$(\frac{1}{10}, \frac{9}{10}, \half, \frac{1}{5})$ &
$(\frac{1}{10}, \frac{9}{10}, \half, \frac{4}{5})$ & 
$(\frac{1}{10}, \frac{9}{10}, \frac{1}{5}, \frac{4}{5})$ &
$(\frac{1}{10}, \frac{13}{30}, \half, \frac{1}{3})$ &
$(\frac{1}{10}, \frac{13}{30}, \half, \frac{1}{5})$ \\

$(\frac{1}{12}, \frac{3}{4}, \half, \half)$ & 
$(\frac{1}{12}, \frac{3}{4}, \half, \frac{1}{3})$ & 
$(\frac{1}{12}, \frac{3}{4}, \frac{1}{3}, \half)$ & 
$(\frac{1}{12}, \frac{5}{6}, \half, \frac{2}{3})$ &
$(\frac{1}{12}, \frac{5}{6}, \half, \frac{1}{4})$ &
$(\frac{1}{12}, \frac{5}{6}, \frac{1}{4}, \frac{2}{3})$ \\

$(\frac{1}{12}, \frac{5}{12}, \half, \frac{1}{4})$ &
$(\frac{1}{12}, \frac{7}{12}, \half, \frac{1}{3})$ &
$(\frac{1}{12}, \frac{7}{12}, \frac{1}{3}, \half)$ &
$(\frac{1}{12}, \frac{11}{12}, \frac{1}{3}, \half)$ &
$(\frac{1}{15}, \frac{7}{15}, \half, \frac{1}{3})$ & 
$(\frac{1}{15}, \frac{7}{15}, \half, \frac{1}{5})$ \\

$(\frac{1}{15}, \frac{11}{15}, \half, \frac{1}{5})$ &
$(\frac{1}{15}, \frac{11}{15}, \half, \frac{3}{5})$ &
$(\frac{1}{15}, \frac{13}{15}, \half, \frac{1}{3})$ &
$(\frac{1}{15}, \frac{13}{15}, \half, \frac{3}{5})$ &
$(\frac{1}{20}, \frac{11}{20}, \half, \frac{1}{5})$ & 
$(\frac{1}{20}, \frac{11}{20}, \half, \frac{2}{5})$ \\

$(\frac{1}{20}, \frac{13}{20}, \half, \half)$ &
$(\frac{1}{20}, \frac{13}{20}, \half, \frac{1}{5})$ & 
$(\frac{1}{20}, \frac{17}{20}, \half, \half)$ &
$(\frac{1}{20}, \frac{17}{20}, \half, \frac{2}{5})$ &
$(\frac{1}{24}, \frac{13}{24}, \half, \frac{1}{3})$ &
$(\frac{1}{24}, \frac{13}{24}, \half, \frac{1}{4})$ \\
 
$(\frac{1}{24}, \frac{17}{24}, \half, \half)$ &
$(\frac{1}{24}, \frac{17}{24}, \half, \frac{1}{4})$ &
$(\frac{1}{24}, \frac{19}{24}, \half, \half)$ & 
$(\frac{1}{24}, \frac{19}{24}, \half, \frac{1}{3})$ &
$(\frac{1}{30}, \frac{5}{6}, \half, \frac{2}{3})$ &
$(\frac{1}{30}, \frac{5}{6}, \half, \frac{1}{5})$ \\

$(\frac{1}{30}, \frac{5}{6}, \frac{1}{5}, \frac{2}{3})$ &
$(\frac{1}{30}, \frac{7}{10}, \half, \frac{1}{3})$ &
$(\frac{1}{30}, \frac{7}{10}, \half, \frac{2}{5})$ &
$(\frac{1}{30}, \frac{7}{10}, \frac{1}{3}, \frac{2}{5})$ &
$(\frac{1}{30}, \frac{11}{30}, \half, \frac{1}{5})$ &
$(\frac{1}{30}, \frac{19}{30}, \half, \frac{1}{3})$ \\

$(\frac{1}{60}, \frac{31}{60}, \half, \frac{1}{3})$ &
$(\frac{1}{60}, \frac{31}{60}, \half, \frac{1}{5})$ &
$(\frac{1}{60}, \frac{41}{60}, \half, \half)$ &
$(\frac{1}{60}, \frac{41}{60}, \half, \frac{1}{5})$ &
$(\frac{1}{60}, \frac{49}{60}, \half, \half)$ &
$(\frac{1}{60}, \frac{49}{60}, \half, \frac{1}{3})$ \\
\hline
\end{tabularx}
\end{center}
\end{table}
\renewcommand{\arraystretch}{1}


\subsection{The Horn $H_5$ function}\label{subsec:H5}

The $H_5$ function is defined by 
\begin{equation*}
H_5(a,b,c|x,y) = \sum_{m,n \geq 0} \frac{\poch{a}{2m+n} \poch{b}{n-m}}{\poch{c}{n} m! n!} x^m y^n. 
\end{equation*}
Hence the lattice is $\lat= \Z (-2, 1, 0, 1, 0) \oplus \Z (-1, -1, 1, 0, 1)$.
We can take $\A = \{\e_1, \e_2, \e_3, 2\e_1-\e_2, \e_1+\e_2-\e_3\}$ and $\gam = (-a,-b,c-1,0,0)$, so $\bal=(-a,-b,c-1)$.

The projection of $\A$ onto the $(x_1,x_2)$-plane is shown in Figure~\ref{fig:H5_A}. 
The thick dots represent $\A$, the dark gray region is the set $Q(\A)$ and light gray region is a part of the set $C(\A)$.
It is clear that $Q(\A)$ has a unimodular triangulation and has volume 4, so $\A$ is saturated.
Furthermore,
\begin{equation*}
C(\A) = \{\x \in \R^3 \ | \ x_1 \geq 0, x_1+2x_2 \geq 0, x_1+x_3 \geq 0, x_1+2x_2+3x_3 \geq 0\}.
\end{equation*}
Hence $H_5(a,b,c|x,y)$ is non-resonant if and only if $a$, $a+2b$, $c-a$ and $3c-a-2b$ are non-integral.

\begin{figure}[htb]
\begin{center}
\psset{unit=0.6cm}
\begin{picture}(3,4.5)(-1,-2)
\pspolygon[linestyle=none,fillstyle=solid,fillcolor=lightgray](0,0)(3,-1.5)(3,2)(0,2)
\pspolygon[linewidth=2pt,fillstyle=solid,fillcolor=darkgray](0,0)(2,-1)(1,1)(0,1)
\psline(0,2)(0,0)(3,-1.5)
\psline[linewidth=2pt](0,1)(2,-1)
\psline[linewidth=2pt](0,0)(1,0)(1,1)
\psdots[dotsize=5pt](1,0)(0,1)(0,0)(2,-1)(1,1)
\psgrid[subgriddiv=1,griddots=1,gridwidth=2pt,gridlabels=0](0,0)(-1,-2)(3,2)
\psaxes[ticks=none,linewidth=0.5pt,labels=none](0,0)(-1,-2)(3,2)
\end{picture}
\caption{The sets $\A$, $Q(\A)$ and $C(\A)$ for $H_5$}
\label{fig:H5_A}
\psset{unit=1cm}
\end{center}
\end{figure}

\begin{lem}\label{lem:H5_interlacing}
There are 4 apexpoints if and only if either $\fp{a} \leq \fp{c}$, $1 < \fp{a}+2\fp{b} \leq 2$ and $3\fp{c} < \fp{a}+2\fp{b} \leq 3\fp{c}+1$, or $\fp{c} < \fp{a}$, $1 < \fp{a}+2\fp{b} \leq 2$ and $3\fp{c}-1 < \fp{a}+2\fp{b} \leq 3\fp{c}$. 
\end{lem}

\begin{proof}
There are 4 apexpoints if and only if $(\entier{\al_1+2\al_2}, \entier{\al_1+\al_3}, \entier{\al_1+2\al_2+3\al_3})$ equals $(1,0,3)$ or $(1,1,2)$.
In those cases, $\al_2$ is non-integral, so $\bal=(1-\fp{a}, 1-\fp{b}, \fp{c})$.
\end{proof}

\begin{thm}\label{H5_solutions}
$H_5(a,b,c|x,y)$ is non-resonant and algebraic if and only if $(a,b,c)$ is, up to equivalence modulo $\Z$, conjugate to one of the following:
$(r,-r,\half)$ for some $r \in \parset$,
$(\frac{1}{6}, \half, \frac{1}{3})$,
$(\frac{1}{6}, \frac{2}{3}, \frac{1}{3})$,
$(\frac{1}{6}, \frac{5}{6}, \frac{1}{3})$,
$(\frac{1}{10}, \frac{3}{5}, \frac{1}{5})$ and
$(\frac{1}{12}, \frac{3}{4}, \half)$.
\end{thm}

\begin{proof}
Suppose that $H_5(a,b,c|x,y)$ is non-resonant and algebraic.
Then $F(a,b,c|z) = H_5(a,b,c|0,z)$ is also algebraic and irreducible, so $(a,b,c)$ is a Gauss triple.

First suppose that $(a,b,c)$ is a Gauss triple of type 1.
If $(a,b,c) = (r,-r,\half)$, then the function is non-resonant.
If $r<\half$, then $1 < \fp{a}+2\fp{b} \leq 2$ and $3\fp{c} < \fp{a}+2\fp{b} \leq 3\fp{c}+1$, and if $r>\half$, then $\fp{c} < \fp{a}$, $1 < \fp{a}+2\fp{b} \leq 2$ and $3\fp{c}-1 < \fp{a}+2\fp{b} \leq 3\fp{c}$.
Hence the interlacing condition is satisfied.
Since all conjugates are of the same form, the function is algebraic.

Suppose that $(a,b,c) = (r, r+\half, \half)$.
The function is non-resonant if $r$ is not equal to $\frac{1}{3}$, $\frac{2}{3}$, $\frac{1}{6}$ or $\frac{5}{6}$.
Then the interlacing condition is satisfied if and only if $r \in (\frac{1}{6}, \frac{1}{3}] \cup (\frac{2}{3},\frac{5}{6}]$.
All conjugates of $r$ also have to be in this set.
By choosing a conjugate with numerator 1, we get that the denominator of $r$ can at most be 5.
$\frac{1}{4}$ and $\frac{3}{4}$ give the same tuple as for $(r,-r,\half)$.
If $r$ has denominator 5, then $\frac{2}{5}$ is a conjugate that doesn't satisfy the condition.
Hence this gives no extra algebraic functions.

Finally, suppose that $(a,b,c) = (r, r+\half, 2r)$.
Then the function is non-resonant if $r$ is not equal to $\frac{1}{3}$ or $\frac{2}{3}$.
The interlacing condition is satisfied if and only if $r \in (0,\frac{1}{3}) \cup (\frac{2}{3},1)$.
By Lemma~\ref{lem:closetohalf}, this implies that the denominator of $r$ is 4, 6 or 10.
If the denominator is 4, then we again get the tuple $(r,-r,\half)$.
There are two solutions with denominator 6: $(\frac{1}{6}, \frac{2}{3}, \frac{1}{3})$ and $(\frac{5}{6}, \frac{1}{3}, \frac{2}{3})$.
With denominator 10, we find the solutions $(\frac{1}{10}, \frac{3}{5}, \frac{1}{5})$, $(\frac{3}{10}, \frac{4}{5}, \frac{3}{5})$, $(\frac{7}{10}, \frac{1}{5}, \frac{2}{5})$ and $(\frac{9}{10}, \frac{2}{5}, \frac{4}{5})$.
For all these tuples, the interlacing condition is indeed satisfied.

If $(a,b,c)$ is a Gauss triple of type 2, then there are only finitely many possibilities.
There are 8 solutions: 
$\pm(\frac{1}{6}, \half, \frac{1}{3})$,
$\pm(\frac{1}{6}, \frac{5}{6}, \frac{1}{3})$,
$\pm(\frac{1}{12}, \frac{3}{4}, \half)$ and
$\pm(\frac{5}{12}, \frac{3}{4}, \half)$.
\end{proof}


\subsection{The Horn $H_6$ function}\label{sec:H6}

The $H_6$ function is defined by 
\begin{equation*}
H_6(a,b,c|x,y) = \sum_{m,n \geq 0} \frac{\poch{a}{2m-n} \poch{b}{n-m} \poch{c}{n}}{m! n!} x^m y^n. 
\end{equation*}
Hence the lattice is $\lat= \Z (-2, 1, 0, 1, 0) \oplus \Z (1, -1, -1, 0, 1)$.
We choose $\A = \{\e_1, \e_2, \e_3, 2\e_1-\e_2, -\e_1+\e_2+\e_3\}$ and $\gam=(-a, -b, -c, 0, 0)$.
Then $\bal=(-a, -b, -c)$.
The function $f: x \mapsto (x_1-x_2, x_2, x_2+x_3)$ maps $G_1$ to $H_6$.
Hence:

\begin{lem}\label{lem:H6_irreducible}
$H_6(a,b,c|x,y)$ is non-resonant if and only if $a+b$, $a+2b$, $c$ and $a+c$ are non-integral.
\end{lem}

\begin{thm}\label{thm:H6_solutions}
$H_6(a,b,c|x,y)$ is non-resonant and algebraic if and only if up to equivalence modulo $\Z$, $(a,b,c)$ equals
$\pm(\half, \frac{1}{3}, \frac{2}{3})$,
$\pm(\half, \frac{1}{3}, \frac{5}{6})$ or
$\pm(\frac{1}{3}, \half, \frac{5}{6})$.
\end{thm}


\subsection{The Horn $H_7$ function}\label{subsec:H7}

The $H_7$ function is defined by
\begin{equation*}
H_7(a,b,c,d|x,y) = \sum_{m,n \geq 0} \frac{\poch{a}{2m-n} \poch{b}{n} \poch{c}{n}}{\poch{d}{m} m! n!} x^m y^n.  
\end{equation*}
The lattice is $\lat = \Z (-2, 0, 0, 1, 1, 0) \oplus \Z (1, -1, -1, 0, 0, 1)$.
We can take $\A = \{\e_1, \e_2, \e_3, \e_4, 2\e_1-\e_4, -\e_2+\e_2+\e_3\}$ and $\gam = (-a,-b,-c,d-1,0,0)$.
Then $\bal=(-a,-b,-c,d-1)$.

The function $f: x \mapsto (x_1-x_2, x_2, x_2+x_4, x_3)$ maps $H_4$ to $H_7$.
Hence:

\begin{lem}\label{lem:H7_irreducible}
$H_7(a,b,c,d|x,y)$ is non-resonant if and only if $b$, $c$, $a+b$, $a+c$, $2d-a-b$ and $2d-a-c$ are non-integral.
\end{lem}

\begin{thm}\label{H7_solutions}
$H_7(a,b,c,d|x,y)$ is non-resonant and algebraic if and only if at least one of $(a,b,c,d)$ or $(a,c,b,d)$ is conjugate to one of the tuples in Table~\ref{tab:H7_solutions}.
\end{thm}

\renewcommand{\arraystretch}{2}
\begin{table}[htb]
\begin{center}
\caption{The tuples $(a,b,c,d)$ such that $H_7(a,b,c,d|x,y)$ is non-resonant and algebraic} 
\label{tab:H7_solutions} 
\begin{tabularx}{\textwidth}{XXXXXX}
\hline
$(\half,r,-r,\half)$ &
$(\half,r,r+\half,\half)$ &
\multicolumn{2}{l}{$(-2r,r,r+\half,\half)$} &
\multicolumn{2}{l}{with $r \in \parset$} \\

$(\half, \frac{1}{4}, \frac{7}{12}, \half)$ &
$(\half, \frac{1}{12}, \frac{7}{12}, \frac{1}{3})$ &
$(\half, \frac{1}{20}, \frac{13}{20}, \half)$ &
$(\half, \frac{1}{24}, \frac{17}{24}, \half)$ &
$(\half, \frac{1}{24}, \frac{19}{24}, \half)$ &
$(\half, \frac{1}{60}, \frac{41}{60}, \half)$ \\

$(\half, \frac{1}{60}, \frac{49}{60}, \half)$ &
$(\frac{1}{3}, \half, \frac{5}{6}, \half)$ &
$(\frac{1}{3}, \frac{1}{4}, \frac{3}{4}, \half)$ &
$(\frac{1}{3}, \frac{1}{4}, \frac{3}{4}, \frac{1}{3})$ &
$(\frac{1}{3}, \frac{1}{4}, \frac{11}{12}, \half)$ &
$(\frac{1}{3}, \frac{1}{6}, \frac{5}{6}, \half)$ \\
 
$(\frac{1}{3}, \frac{1}{6}, \frac{5}{6}, \frac{1}{3})$ &
$(\frac{1}{3}, \frac{5}{6}, \frac{1}{12}, \half)$ &
$(\frac{1}{3}, \frac{5}{6}, \frac{1}{30}, \half)$ &
$(\frac{1}{3}, \frac{1}{10}, \frac{9}{10}, \half)$ &
$(\frac{1}{3}, \frac{1}{10}, \frac{9}{10}, \frac{1}{3})$ &
$(\frac{1}{3}, \frac{1}{10}, \frac{23}{30}, \half)$ \\

$(\frac{1}{3}, \frac{5}{12}, \frac{11}{12}, \half)$ &
$(\frac{1}{3}, \frac{2}{15}, \frac{11}{15}, \half)$ &
$(\frac{1}{3}, \frac{5}{24}, \frac{17}{24}, \half)$ &
$(\frac{1}{3}, \frac{5}{24}, \frac{23}{24}, \half)$ &
$(\frac{1}{3}, \frac{11}{30}, \frac{29}{30}, \half)$ &
$(\frac{1}{3}, \frac{11}{60}, \frac{41}{60}, \half)$ \\

$(\frac{1}{3}, \frac{11}{60}, \frac{59}{60}, \half)$ &
$(\frac{1}{4}, \frac{1}{6}, \frac{5}{6}, \half)$ &
$(\frac{1}{4}, \frac{1}{6}, \frac{5}{6}, \frac{1}{4})$ &
$(\frac{1}{4}, \frac{1}{6}, \frac{11}{12}, \half)$ &
$(\frac{1}{4}, \frac{7}{12}, \frac{11}{12}, \half)$ &
$(\frac{1}{4}, \frac{7}{24}, \frac{19}{24}, \half)$ \\

$(\frac{1}{4}, \frac{7}{24}, \frac{23}{24}, \half)$ &
$(\frac{1}{5}, \frac{1}{6}, \frac{5}{6}, \half)$ &
$(\frac{1}{5}, \frac{1}{6}, \frac{5}{6}, \frac{1}{5})$ &
$(\frac{1}{5}, \frac{1}{6}, \frac{29}{30}, \half)$ &
$(\frac{1}{5}, \frac{1}{10}, \frac{9}{10}, \half)$ &
$(\frac{1}{5}, \frac{1}{10}, \frac{9}{10}, \frac{1}{5})$ \\

$(\frac{1}{5}, \frac{7}{10}, \frac{9}{10}, \half)$ &
$(\frac{1}{5}, \frac{9}{10}, \frac{7}{30}, \half)$ &
$(\frac{1}{5}, \frac{4}{15}, \frac{13}{15}, \half)$ &
$(\frac{1}{5}, \frac{4}{15}, \frac{14}{15}, \half)$ &
$(\frac{1}{5}, \frac{8}{15}, \frac{13}{15}, \half)$ &
$(\frac{1}{5}, \frac{7}{20}, \frac{17}{20}, \half)$ \\

$(\frac{1}{5}, \frac{7}{20}, \frac{19}{20}, \half)$ &
$(\frac{1}{5}, \frac{9}{20}, \frac{19}{20}, \half)$ &
$(\frac{1}{5}, \frac{19}{30}, \frac{29}{30}, \half)$ &
$(\frac{1}{5}, \frac{19}{60}, \frac{49}{60}, \half)$ &
$(\frac{1}{5}, \frac{19}{60}, \frac{59}{60}, \half)$ &
$(\frac{1}{6}, \frac{5}{12}, \frac{11}{12}, \frac{1}{3})$ \\
\hline
\end{tabularx}
\end{center}
\end{table}
\renewcommand{\arraystretch}{1}


\bibliography{authors,journalsfullname,journalsabbreviations,bibliography}

\end{document}